\documentclass[12 pt,oneside,reqno]{article}
\usepackage[english]{babel}

\usepackage[letterpaper,top=2cm,bottom=2cm,left=3cm,right=3cm,marginparwidth=1.75cm]{geometry}

\usepackage[dvips]{graphicx}
\usepackage{calc}
\usepackage{color}
\usepackage{amsmath}
\usepackage{amssymb}
\usepackage{amscd}
\usepackage{amsthm}
\usepackage{amsbsy}
\usepackage{delarray}
\usepackage{tikz}
\usepackage{tikz-cd}
\usepackage{enumerate}
\usepackage{enumitem}
\usepackage[T1]{fontenc}
\usepackage{enumerate}
\usepackage{stackengine} 
\usepackage{mathrsfs}
\usepackage{hyperref}
\usepackage[whole]{bxcjkjatype}
\usepackage{titlesec}
\usepackage[nottoc]{tocbibind}
\usepackage{dynkin-diagrams}
\usepackage{etoolbox}
\usepackage{dynkin-diagrams}
\def\row#1/#2!{#1_{\IfStrEq{#2}{}{n}{#2}} & \dynkin{#1}{#2}\\}
\newcommand{\tble}[1]{
   \renewcommand*\do[1]{\row##1!}
   \[
      \begin{array}{ll}\docsvlist{#1}\end{array}
   \]
}

\usepackage{lipsum}

\definecolor{YK}{rgb}{9,0,0}
\definecolor{YKb}{rgb}{0,0,9}

 \date{\vspace{-4ex}}

\begin{document}



\setlength{\parindent}{5mm}
\renewcommand{\leq}{\leqslant}
\renewcommand{\geq}{\geqslant}
\newcommand{\N}{\mathbb{N}}
\newcommand{\sph}{\mathbb{S}}
\newcommand{\Z}{\mathbb{Z}}
\newcommand{\R}{\mathbb{R}}
\newcommand{\Q}{\mathbb{Q}}
\newcommand{\C}{\mathbb{C}}
\newcommand{\curlL}{\mathcal{L}}
\newcommand{\A}{\mathcal{A}}
\newcommand{\curlP}{\mathcal{P}}
\newcommand{\F}{\mathbb{F}}
\newcommand{\g}{\mathfrak{g}}
\newcommand{\h}{\mathfrak{h}}
\newcommand{\K}{\mathbb{K}}
\newcommand{\RN}{\mathbb{R}^{2n}}
\newcommand{\ci}{c^{\infty}}
\newcommand{\derive}[2]{\frac{\partial{#1}}{\partial{#2}}}
\renewcommand{\S}{\mathbb{S}}
\renewcommand{\H}{\mathbb{H}}
\newcommand{\eps}{\varepsilon}
\newcommand{\lag}{Lagrangian}
\newcommand{\sub}{submanifold}
\newcommand{\homo}{homogeneous}
\newcommand{\qmor}{quasimorphism}
\newcommand{\enum}{enumerate}
\newcommand{\sa}{symplectically aspherical}
\newcommand{\ovl}{\overline}
\newcommand{\wt}{\widetilde}
\newcommand{\hamd}{Hamiltonian diffeomorphism}
\newcommand{\QH}{quantum cohomology ring}
\newcommand{\thm}{Theorem}
\newcommand{\cor}{Corollary}
\newcommand{\hamil}{Hamiltonian}
\newcommand{\propo}{Proposition}
\newcommand{\conjec}{Conjecture}
\newcommand{\asympt}{asymptotic}
\newcommand{\PR}{pseudo-rotation}
\newcommand{\fuk}{\mathscr{F}}
\newcommand{\specinv}{spectral invariant}
\newcommand{\wrt}{with respect to}
\newcommand{\suphv}{superheavy}
\newcommand{\suphvness}{superheaviness}
\newcommand{\symp}{symplectic}
\newcommand\oast{\stackMath\mathbin{\stackinset{c}{0ex}{c}{0ex}{\ast}{\bigcirc}}}
\newcommand{\diffeo}{diffeomorphism}
\newcommand{\quasi}{quasimorphism}
\newcommand{\univham}{\widetilde{\Ham}}
\newcommand{\BK}{Biran--Khanevsky}
\newcommand{\BC}{Biran--Cornea}
\newcommand{\TV}{Tonkonog--Varolgunes}
\newcommand{\varol}{Varolgunes}
\newcommand{\mfd}{manifold}
\newcommand{\smfd}{submanifold}
\newcommand{\EP}{Entov--Polterovich}
\newcommand{\pol}{Polterovich}
\newcommand{\suppot}{superpotential}
\newcommand{\CY}{Calabi--Yau}
\newcommand{\degen}{degenerate}
\newcommand{\coeff}{coefficient}
\newcommand{\acs}{almost complex structure}
\newcommand{\prl}{pearl trajectory}
\newcommand{\prls}{pearl trajectories}
\newcommand{\holo}{holomorphic}
\newcommand{\trans}{transversality}
\newcommand{\pert}{perturbation}
\newcommand{\idem}{idempotent}
\newcommand{\critpt}{critical point}
\newcommand{\tordeg}{toric degeneration}
\newcommand{\nbhd}{neighborhood}
\newcommand{\polar}{polarization}
\newcommand{\symplecto}{symplectomorphism}
\newcommand{\symplectotic}{symplectomorphic}
\newcommand{\FOOO}{Fukaya--Oh--Ohta--Ono}
\newcommand{\quadr}{quadric hypersurface}
\newcommand{\monoconst}{monotonicity constant}
\newcommand{\constr}{construction}
\newcommand{\coord}{coordinate}
\newcommand{\decomp}{decomposition}
\newcommand{\hypsurf}{hypersurface}
\newcommand{\NNU}{Nishinou--Nohara--Ueda}
\newcommand{\transv}{transversality}
\newcommand{\fibr}{fibration}
\newcommand{\atf}{almost toric fibration}
\newcommand{\degtion}{degeneration}
\newcommand{\config}{configuration}
\newcommand{\sing}{singular}
\newcommand{\lagsph}{Lagrangian sphere}
\newcommand{\splgen}{split generation}
\newcommand{\wlg}{without loss of generality}
\newcommand{\AG}{algebraic geometry}
\newcommand{\linindep}{linearly independent}
\newcommand{\RDP}{rational double point}
\newcommand{\grass}{Grassmannian}

\theoremstyle{plain}
\newtheorem{theo}{Theorem}
\newtheorem{theox}{Theorem}
\renewcommand{\thetheox}{\Alph{theox}}
\numberwithin{theo}{subsection}
\newtheorem{prop}[theo]{Proposition}
\newtheorem{lemma}[theo]{Lemma}
\newtheorem{definition}[theo]{Definition}
\newtheorem*{notation*}{Notation}
\newtheorem*{notations*}{Notations}
\newtheorem{corol}[theo]{Corollary}
\newtheorem{conj}[theo]{Conjecture}
\newtheorem{guess}[theo]{Guess}
\newtheorem{claim}[theo]{Claim}
\newtheorem{question}[theo]{Question}
\newtheorem{prob}[theo]{Problem}
\numberwithin{equation}{subsection}

\newenvironment{demo}[1][]{\addvspace{8mm} \emph{Proof #1.
    ~~}}{~~~$\Box$\bigskip}

\newlength{\espaceavantspecialthm}
\newlength{\espaceapresspecialthm}
\setlength{\espaceavantspecialthm}{\topsep} \setlength{\espaceapresspecialthm}{\topsep}

\newenvironment{example}[1][]{\refstepcounter{theo} 
\vskip \espaceavantspecialthm \noindent \textsc{Example~\thetheo
#1.} }%
{\vskip \espaceapresspecialthm}

\newenvironment{remark}[1][]{\refstepcounter{theo} 
\vskip \espaceavantspecialthm \noindent \textsc{Remark~\thetheo
#1.} }%
{\vskip \espaceapresspecialthm}

\def\bb#1{\mathbb{#1}} \def\m#1{\mathcal{#1}}

\def\momeg{(M,\omega)}
\def\co{\colon\thinspace}
\def\Homeo{\mathrm{Homeo}}
\def\Diffeo{\mathrm{Diffeo}}
\def\Symp{\mathrm{Symp}}
\def\Sympeo{\mathrm{Sympeo}}
\def\id{\mathrm{id}}
\newcommand{\norm}[1]{||#1||}
\def\Ham{\mathrm{Ham}}
\def\lagham#1{\mathcal{L}^\mathrm{\Ham}({#1})}
\def\Hamtilde{\widetilde{\mathrm{\Ham}}}
\def\cOlag#1{\mathrm{Sympeo}({#1})}
\def\Crit{\mathrm{Crit}}
\def\diag{\mathrm{diag}}
\def\Spec{\mathrm{Spec}}
\def\osc{\mathrm{osc}}
\def\Cal{\mathrm{Cal}}
\def\Ker{\mathrm{Ker}}
\def\Hom{\mathrm{Hom}}
\def\FS{\mathrm{FS}}
\def\tor{\mathrm{tor}}
\def\Int{\mathrm{Int}}
\def\PD{\mathrm{PD}}
\def\Spec{\mathrm{Spec}}
\def\momeg{(M,\omega)}
\def\co{\colon\thinspace}
\def\Homeo{\mathrm{Homeo}}
\def\Hameo{\mathrm{Hameo}}
\def\Diffeo{\mathrm{Diffeo}}
\def\Symp{\mathrm{Symp}}
\def\Sympeo{\mathrm{Sympeo}}
\def\id{\mathrm{id}}
\def\Im{\mathrm{Im}}
\def\Ham{\mathrm{Ham}}
\def\lagham#1{\mathcal{L}^\mathrm{Ham}({#1})}
\def\Hamtilde{\widetilde{\mathrm{Ham}}}
\def\cOlag#1{\mathrm{Sympeo}({#1})}
\def\Crit{\mathrm{Crit}}
\def\dim{\mathrm{dim}}
\def\Spec{\mathrm{Spec}}
\def\osc{\mathrm{osc}}
\def\Cal{\mathrm{Cal}}
\def\Fix{\mathrm{Fix}}
\def\det{\mathrm{det}}
\def\Ker{\mathrm{Ker}}
\def\coker{\mathrm{coker}}
\def\Per{\mathrm{Per}}
\def\rank{\mathrm{rank}}
\def\Span{\mathrm{Span}}
\def\Supp{\mathrm{Supp}}
\def\Hof{\mathrm{Hof}}
\def\grad{\mathrm{grad}}
\def\ind{\mathrm{ind}}
\def\Hor{\mathrm{Hor}}
\def\Vert{\mathrm{Vert}}
\def\Re{\mathrm{Re}}

\title{Isolated hypersurface singularities, spectral invariants, and quantum cohomology}
\author{Yusuke Kawamoto}

\newcommand{\Addresses}{{
  \bigskip
  \footnotesize

   \textsc{Yusuke Kawamoto, Institute for Mathematical Research (FIM), R\"amistrasse 101, 8092 Z\"urich Switzerland}\par\nopagebreak
  \textit{E-mail address}: \texttt{yusukekawamoto81@gmail.com, yusuke.kawamoto@math.ethz.ch} }}

\maketitle

\begin{abstract}
We study the relation between isolated hypersurface singularities (e.g. ADE) and the quantum cohomology ring by using spectral invariants, which are symplectic measurements coming from Floer theory. We prove, under the assumption that the quantum cohomology ring is semi-simple, that (1) if the smooth Fano variety degenerates to a Fano variety with an isolated hypersurface singularity, then the singularity has to be an $A_m$-singularity, (2) if the symplectic manifold contains an $A_m$-configuration of Lagrangian spheres, then there are consequences for the Hofer geometry, and that (3) the Dehn twist reduces spectral invariants.
\end{abstract}

\tableofcontents

\section{Introduction}

\subsection{Context}\label{context}


Degeneration is a theme that originates in classical algebraic geometry that is still actively studied in the context of various modern topics such as the minimal model program \cite{[KM98]}, K\"ahler--Einstein metric \cite{[DK01]}, mirror symmetry, and the SYZ conjecture \cite{[Gr13]}. Its importance in {\symp} topology was noticed by Arnold \cite{[Arn95]} and Donaldson \cite{[Don00]}, especially that {\lag}s can appear as vanishing cycles. Seidel largely developed this idea and obtained various results on the {\symp} aspect of the Dehn twist \cite{[Sei97],[Sei99],[Sei00],Sei08}. In this paper, we study degeneration by {\symp} topology, and vice versa. Note that such an attempt was also made by Biran which was highlighted in his ICM address \cite[Section 5.2]{[Bir02]}. 

Understanding the type of {\sing}ities an algebraic variety can {\degen} to is an important subject in algebraic geometry. \textit{Isolated hypersurface {\sing}ities} have been fundamental in the study of {\sing}ities since \cite{[Mil68]}. Any isolated hypersurface {\sing}ity can be associated to a positive integer called \textit{modality}, which roughly speaking, expresses the complexity of the {\sing}ity. Arnold classified isolated hypersurface {\sing}ities up to modality two \cite{[Arn76],[AGLV93]} and according to his classification, the ones with modality zero are called the \textit{simple (ADE) {\sing}ities}. Up to right equivalence, $A_m$, $D_m\ (m\geq 4)$, and $E_m\ (m=6,7,8)$ {\sing}ities are given as
$$x_1 ^2 + x_2 ^2 + \cdots + x_{n-1} ^2+ p(y,z)=0$$
where $p(y,z)$ is 
$$y^2 + z^{m+1},\ 
     y z ^2 +     z ^{m-1} ,\
 y ^3  +z ^{4} ,\
y ^3 + y z ^{3} ,\
 y^3 + z ^{5} ,$$
respectively. The vanishing cycle of an A, D, E {\sing}ity forms a {\config} of {\lagsph}s that intersect as expressed in the A, D, E type Dynkin diagram (Figure \ref{ADE dynkin diagram}), respectively. We call these {\config}s of {\lagsph}s ADE {\config}s.

\begin{figure}[h]
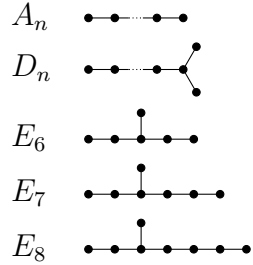

    \tble{A/{},D/{},E/6,E/7,E/8}
    \caption{Dynkin diagrams of type $A_{n},D_{n},E_6,E_7,E_8$.}
    \label{ADE dynkin diagram}
\end{figure}

 Isolated hypersurface {\sing}ities of modality one consists of three types, namely the \textit{parabolic (or simple elliptic) {\sing}ities} $\wt{E}_{6},\wt{E}_{7},\wt{E}_{8}$, \textit{hyperbolic {\sing}ities} $T_{p,q,r}$, and \textit{14 exceptional {\sing}ities}. Similarly to the simple {\sing}ities, all the isolated hypersurface {\sing}ities of positive modality also give rise to {\config}s of {\lagsph}s by taking the vanishing cycles. We refer to Section \ref{hypsing} for further information.

Note that for surfaces, i.e. complex dimension two, simple {\sing}ities have many different characterizations such as du Val {\sing}ities, rational double points, Kleinian {\sing}ities \cite{[Rei]}. For this case, algebraic geometers have a fairly good understanding of {\degtion}s. In fact, for Fano surfaces, i.e. the del Pezzo surfaces, du Val classified all the possible simple singularities that can occur on {\sing} del Pezzo surfaces \cite{[DV34]}. On the other hand, very little is known for higher dimensional spaces, and the importance of studying the higher dimensional case is emphasized by Arnold in \cite{[Arn95]}.

Another object that has been of interest to both algebraic and {\symp} geometers is the \textit{quantum cohomology}. After its introduction in string theory by Vafa and Witten \cite{[Vaf91],[Wit91]}, the algebro-geometric formulation was found by Kontsevich--Manin in \cite{[KM94]}, shortly followed by the {\symp} formulation in \cite{[RT95]} due to Ruan--Tian. An important case is when the {\QH} is semi-simple; the (small) quantum cohomology ring\footnote{In this paper, {\QH} will always refer to the small one unless mentioned otherwise. See Section \ref{Semi-simplicity of QH} for further comments on different notions of semi-simplicity for small and big {\QH}s.} $QH(X,\omega)$ of a {\symp} {\mfd} $(X,\omega)$ is semi-simple if it splits into a direct sum of finitely many fields $\{Q_j\}_{1 \leq j \leq k}$:
$$QH(X,\omega) = \bigoplus_{1 \leq j \leq k} Q_j.$$
Examples of {\symp} {\mfd}s equipped with monotone {\symp} forms\footnote{Recall that a {\symp} {\mfd} $(X,\omega)$ is monotone if we have
$[\omega]|_{\pi_2(X)} = \kappa_X \cdot c_1(TX)|_{\pi_2(X)}$ for some $\kappa_X>0$.} that have semi-simple {\QH}s over $\C$-{\coeff}s\footnote{Note that semi-simplicity depends on the choice of {\coeff}s. For example, over $\F_p$, a field of characteristic $p$, the {\QH} of $\C P^n$ is not semi-simple when $p$ divides $n+1$. In this paper, we work over $\C$-{\coeff}s.} include 
\begin{itemize}
    \item the complex projective space $\C P^n$ \cite{[MS04]},
    \item the quadric hypersurfaces $Q^n$ \cite{[Abr00]},
    \item the del Pezzo surfaces $\mathbb{D}_k:=\C P^2 \# k ( \overline{\C P^2})$ with $0 \leq k \leq 4$ (where $k$ is the number of the points blown up in such a way that the {\symp} {\mfd} become monotone) \cite{[BM04],[CM95]},
    \item  the complex Grassmannians $Gr_{\C} (k,n)$ (i.e. the space of complex $k$ dimensional linear subspace in $\C ^n$) \cite{[Abr00]},
    \item some homogeneous spaces \cite{[CMP10],[Per]} and some generalized {\grass}s \cite{[Gra]},
    \item products of monotone {\symp} {\mfd}s with semi-simple {\QH}s \cite{[EP08]}.
\end{itemize}

In addition to the above, examples of {\symp} {\mfd}s equipped with generic {\symp} forms that have semi-simple {\QH}s include 

\begin{itemize}
    \item any {\symp} toric Fano {\mfd} \cite{[FOOO10],[OT09],[Ush11]},

    \item 36 out of the 59 Fano 3-folds with no odd rational cohomology \cite{[Cio05]},

    \item one-point blow up of any of the above \cite{[Ush11]}.
\end{itemize}

   The semi-simplicity of the (small/big) {\QH} has also shared interest in algebraic and {\symp} geometry; for example, see \cite{[BM04],[Dub96], [KM94]} for the algebraic geometry side, \cite{[EP03],[EP08],[EP09],[BC],[FOOO10]} for the symplectic side. See Section \ref{Semi-simplicity of QH} for more information on the above examples and the relation between different notions of semi-simplicity. The semi-simplicity also has important implication for physics; we refer the readers to \cite{[HKK$^+$]}.

In this paper, we study the interaction between isolated {\hypsurf} singularities and the semi-simplicity of quantum cohomology rings, which are both objects lying in the intersection of algebraic and {\symp} geometry. Our method is based on Floer theory, more precisely the theory of {\specinv}s, which allows us to have less dimensional restrictions than in the current algebraic geometry.

\subsection{Isolated {\hypsurf} {\sing}ities and quantum cohomology}

We have seen in the previous section that the theories of {\sing}ities and the quantum cohomology are both of interest to algebraic and symplectic geometers. However, the interaction between the two theories has not been studied.\footnote{The author thanks Kaoru Ono and Rahul Pandharipande for pointing this out.} Our first main result, which can be formulated in terms of algebraic and {\symp} geometry, claims that semi-simplicity of the quantum cohomology ring excludes most of the isolated {\hypsurf} {\sing}ities. We first state the algebro-geometric version.

     \begin{theox}[{Algebraic geometry version}]\label{no DE config}
    Let $X$ be a complex $n$ dimensional smooth Fano variety. Assume $QH(X,\omega)$ is semi-simple, where $\omega$ is the anti-canonical form of $X$. If $X$ degenerates to a Fano variety with an isolated {\hypsurf} {\sing}ity, then the {\sing}ity has to be
    \begin{itemize}
        \item  an $A_m$-{\sing}ity with $m\geq 1$, if $n$ is even.

        \item  an $A_m$-{\sing}ity with $m= 1,2$, if $n$ is odd and $\frac{\dim_\C X +1}{2r_X} \notin \Z$ where $r_X$ is the Fano index.
    \end{itemize}
    
    \end{theox}

    \begin{remark}\label{Fano symp form}
    Here are some remarks on {\thm} \ref{no DE config}:
    \begin{enumerate}
        \item A smooth Fano variety $X$ carries a natural {\symp} form called the anti-canonical form which comes from the projective embedding $f: X \hookrightarrow \C P^N$ for some $N\in \N$. From the {\symp} perspective, this is a monotone {\symp} form, see \eqref{Fano is monotone}.

        \item It would be very interesting to study the case of other classes of {\sing}ities, for example cyclic quotient {\sing}ities. See Remark \ref{Cyclic quotient remark}.
    
       \item It would be interesting to study the remaining case, i.e. when $n$ is odd and $\frac{\dim_\C X +1}{2r_X} \in \Z$ where $r_X$ is the Fano index. 

       \item {\thm} \ref{no DE config} is expected to be slightly generalized, see {\conjec} \ref{ideal statement}.
    \end{enumerate}
        
    \end{remark}

To prove {\thm} \ref{no DE config}, we reduce it to the following {\thm} \ref{no DE config SG}, which could be regarded as the {\symp} version of {\thm} \ref{no DE config} (We point out that this translation from algebraic geometry to {\symp} topology is not immediate. See Section \ref{From SG to AG} for further information).

    \begin{theox}[{Symplectic topology version}]\label{no DE config SG}
    Let $(X,\omega)$ be a real $2n$ dimensional closed monotone {\symp} {\mfd}. Assume $QH(X,\omega)$ is semi-simple. Then $(X,\omega)$ does not contain 
    \begin{itemize}
        \item a $D_4$-{\config} of {\lagsph}s, if $n$ is even.
        \item an $A_3$-{\config} of {\lagsph}s,  if $n$ is odd and $\frac{n +1}{2N_X} \notin \Z$ where $N_X$ is the minimal Chern number.
    \end{itemize}
    
    \end{theox}

   As we pointed out in Section \ref{context}, the classification of {\sing}ities for varieties is extremely important in {\AG} and mainly due to the lack of methods, very little is known about the possible {\sing}ity types for the higher dimensional varieties. {\thm} \ref{no DE config} has no dimensional restriction, as our approach is based on Floer theory.

We now look at relation to other works. 

\textit{Simple {\sing}ities on surfaces.} As we mentioned earlier, isolated {\hypsurf} {\sing}ities that can occur on surfaces are fairly well understood. 
\begin{itemize}
    \item For the Fano case (i.e. del Pezzo surfaces), the simple (ADE) {\sing}ities that can occur on singular Fano surfaces (i.e. singular del Pezzo surfaces) were completely classified by du Val \cite{[DV34]} (see also \cite[Section 1]{[Sta21]}). Denote the smooth del Pezzo surface of degree $9-k$ by $\mathbb{D}_k$. According to it, $\mathbb{D}_k$ can {\degen} to a singular Fano surface with D or E type {\sing}ities when $5 \leq k \leq 8$,
while it can only {\degen} to a singular Fano surface with A type {\sing}ities when $0 \leq k \leq 4$. The quantum cohomology ring $QH(\mathbb{D}_k)$ is semi-simple (when $\mathbb{D}_k$ is equipped with a monotone {\symp} form) if and only if $0 \leq k \leq 4$, so this is consistent with {\thm} \ref{no DE config}.

\item 
For the Calabi--Yau case, it is also well-known that D,E and the 14 exceptional {\sing}ities can appear in {\degtion}s of the K3 surface. However, Calabi--Yau {\mfd}s, i.e. $c_1|_{\pi_2}=0$, cannot have semi-simple {\QH}s. Once again, this is consistent with {\thm} \ref{no DE config}.

\end{itemize}

\textit{Compactification of Milnor fibers.}
First of all, notice that  {\thm} \ref{no DE config SG} immediately implies the following.

\begin{corol}\label{cannot cptify}
    The Milnor fiber of an isolated {\hypsurf} {\sing}ity that is not of type A cannot be compactified to a {\symp} {\mfd} with semi-simple {\QH}.
\end{corol}

Milnor fiber is, loosely speaking, a smoothing of the {\sing}ity (see Definition \ref{def Milnor fiber}).

\begin{itemize}
    \item In \cite{[Kea15]}, where Keating studied the {\symp} topology of the Milnor fibers of isolated {\hypsurf} {\sing}ities of positive modality, an important step was to find compactifications of the Milnor fibers to del Pezzo surfaces $\mathbb{D}_k$. In \cite[{\propo} 5.19]{[Kea15]}, Keating proves that the Milnor fibers of $\wt{E}_{6},\wt{E}_{7},\wt{E}_{8}$ (isolated {\hypsurf} {\sing}ities of modality one) can be compactified to $\mathbb{D}_6,\mathbb{D}_7,\mathbb{D}_8$, respectively. This is compatible with {\cor} \ref{cannot cptify}, as $\mathbb{D}_6,\mathbb{D}_7,\mathbb{D}_8$ do not have semi-simple {\QH}s. 

    \item 
It is also well-known that the Milnor fiber of the 14 exceptional {\sing}ities (modality one) can be compactified to the K3 surface \cite{[Pin77],[Nik79],[Dol96]}, see also \cite{[KMU13]}. This is also consistent with {\cor} \ref{cannot cptify}.

\end{itemize}

\subsection{$A_m$ {\sing}ities and Hofer geometry}

{\thm} \ref{no DE config SG} implies that there is a possibility that a {\symp} {\mfd} with semi-simple quantum cohomology ring can contain an $A_m$-{\config}. In fact, this can happen; for example, del Pezzo surfaces $ \mathbb{D}_3$ and $\mathbb{D}_4$ have an $A_2$-{\config} and an $A_4$-{\config}, respectively. If we are in such a situation, we get some implication for the Hofer geometry. Before stating the result, recall that the set of {\hamil} {\diffeo}s, denoted by $\Ham(X,\omega)$, form a group and has a remarkable bi-invariant metric called the Hofer metric \cite{[Hof93]}. The study of geometric properties of the group $\Ham(X,\omega)$ {\wrt} the Hofer metric has been an important subject of the field \cite{[Pol01]}. For readers who are not familiar with the subject, we refer to \cite[Section 1.1]{[Kaw]} for a rapid overview of the aspects that are relevant to this paper. Our second main result is the following.

     \begin{theox}\label{Am config many qmor}
    Let $(X,\omega)$ be a real $2n$ dimensional closed monotone {\symp} {\mfd} with even $n$. Assume $QH(X,\omega)$ is semi-simple. If $(X,\omega)$ contains an $A_m$-{\config}, then there are $m-1$ pairwise distinct {\EP} {\qmor}s on $\wt{\Ham} (X,\omega)$.
    \end{theox}

\begin{remark}
     {\EP} {\qmor}s are special maps on $\wt{\Ham} (X,\omega)$, i.e. the universal cover of $\Ham(X,\omega)$, constructed from Floer theory that have powerful applications to Hofer geometry. See Section \ref{EP qmors and superheaviness} for further information.
\end{remark}

\textit{Application.} The del Pezzo surface $\mathbb{D}_4$ has semi-simple {\QH}, and by combining some {\tordeg} and (complex) 2-dimensional techniques with {\thm} \ref{Am config many qmor}, we obtain the following result on the Hofer geometry for $\mathbb{D}_4$.

    \begin{theo}[Kapovich--{\pol} question, Entov--{\pol}--Py question]\label{D4 four qmors}
    There are four pairwise distinct {\EP} {\qmor}s on $\Ham(\mathbb{D}_4)$. Thus, $ \Ham( \mathbb{D}_4 ) $ admits a quasi-isometric embedding of $\R^4$. Moreover, there are three linearly independent {\qmor}s on $\Ham(\mathbb{D}_4)$ that are both $C^0$ and Hofer-Lipshitz continuous. In particular, the group $ \Ham( \mathbb{D}_4 ) $ is not quasi-isometric to the real line $\mathbb{R}$ {\wrt} the Hofer metric.
    \end{theo}

    \begin{remark}
 The Kapovich--{\pol} question, which asks whether for a closed {\symp} {\mfd}  $(X,\omega)$, the group $\Ham(X,\omega)$ is quasi-isometric to the real line {\wrt} the Hofer metric, has been an important open problem in Hofer geometry for a long time, and at the time of writing, it has been answered in the negative for {\symp} {\mfd}s that satisfy some dynamical condition \cite{Ush13}, the monotone $S^2 \times S^2$ \cite{[FOOO19]} (see also \cite{[EliPol]}), and the $2$-sphere \cite{[CGHS],[PS]} and the del Pezzo surfaces $\mathbb{D}_3, \mathbb{D}_4 $ \cite{[Kaw]}. {\thm} \ref{D4 four qmors} improves what was known about the Kapovich--{\pol}, and the Entov--{\pol}--Py questions for $\mathbb{D}_4$ from \cite[{\thm} C(2),D,E]{[Kaw]}.
\end{remark}

\begin{remark}
    It seems likely that the number of pairwise distinct {\EP} {\qmor}s on $\Ham(\mathbb{D}_4)$ in {\thm} \ref{D4 four qmors} can be improved from four to six, see Remark \ref{six qmor}.
\end{remark}

\subsection{Dehn twist and {\specinv}s}

    The proofs of Theorems \ref{no DE config}, \ref{no DE config SG}, and \ref{Am config many qmor} are based on the theory of {\specinv}s (see Section \ref{prelim specinv} for the definition of {\specinv}s). As mentioned earlier in Section \ref{context}, {\config}s of {\lagsph}s were used by Seidel to study the Dehn twist. Recall that the Dehn twist is a (class of) {\symplecto}(s) that is defined for a {\lagsph}. By using some ingredients of the proof of {\thm} \ref{no DE config SG}, we get the following result which describes the effect of the Dehn twist on {\specinv}s.

\begin{theox}\label{Dehn twist spec inv}
    Let $(X,\omega)$ be a real $2n$ dimensional closed monotone {\symp} {\mfd}. Assume $QH(X,\omega)$ is semi-simple and also either one of the following: 
    \begin{itemize}
        \item $n$ is even,
        \item $n$ is odd and $\frac{n +1}{2N_X} \notin \Z$ where $N_X$ is the minimal Chern number.
    \end{itemize}
     If $(X,\omega)$ contains an $A_2$-{\config} of {\lagsph}s $\{L,L'\}$, then we have
      $$  \ovl{\ell}_{\tau_{L}(L')} (H) \leq \max \{ \ovl{\ell}_{L} (H) , \ovl{\ell}_{L'} (H) \} $$
   for any {\hamil} $H$, where $\tau_{L} $ is the Dehn twist about $L$.
\end{theox}

\begin{remark}
\begin{enumerate}
\item The function $ \ovl{\ell}_{L} $ is the asymptotic {\lag} {\specinv} associated to a {\lag} $L$. For the precise definition, see \eqref{asymp lag spec inv} in Section \ref{prelim specinv}.

\item Strictly speaking, the Dehn twist $\tau_L$ about a {\lagsph} $L$ is usually referred to a class of {\symplecto}s, i.e. $\tau_L \in \pi_0 (\Symp(X,\omega))$. In {\thm} \ref{Dehn twist spec inv}, we consider any representative of the class and denote it by $\tau_L $ by abuse of notation. Any two representatives of the Dehn twist are {\hamil} isotopic and thus the {\specinv}s corresponding to $\ell_{\tau_{L}(L')}$ might have a shift up to the Hofer norm of the {\hamil} isotopy between them \cite[{\propo} 2.6]{[LZ18]}. Nevertheless, the \textit{{\asympt}} {\specinv} $\ovl{\ell}_{\tau_{L}(L')}$ does not depend on the choice of the representative of the Dehn twist and is well-defined.

\item It is not difficult to find examples (e.g. in $\mathbb{D}_4$) where we have a strict inequality in {\thm} \ref{Dehn twist spec inv}.

\end{enumerate}
\end{remark}

To put {\thm} \ref{Dehn twist spec inv} into perspective, following the success of the barcode theory in {\symp} topology, there has been a lot of work on the filtration beyond the level of Floer homology, namely for Fukaya categories, e.g. \cite{BCZ,Amb}. Cone-attaching is a fundamental algebraic operation in the $A_{\infty}$-category theory, which in the context of Fukaya category has a geometric interpretation, namely the {\lag} cobordism, e.g. the Dehn twist of a {\lagsph}. Biran--Cornea studied the filtration of Seidel's Floer-theoretic long exact sequence involving the Dehn twist \cite{BC21}, but the precise impact of the Dehn twist on {\specinv}s was not clear. Thus, {\thm} \ref{Dehn twist spec inv} could be regarded as the first step in the study of {\specinv}s in the filtered $A_{\infty}$-categorical setting.

\subsection{Structure of the paper}

To prove {\thm} \ref{no DE config}, we reduce it to its {\symp} counterpart, namely {\thm} \ref{no DE config SG}, by some elementary algebro-geometric argument in Section \ref{From SG to AG}. {\thm} \ref{no DE config SG} is proven by the \textit{spectral rigidity} (i.e. {\symp} rigidity in terms of {\specinv}s) of {\lagsph}s. The key two lemmas are Lemma \ref{sphere spec inv relation} and Lemma \ref{lemma spheres idemp}: the former describes some spectral rigidity of a {\lagsph}, and the latter describes a property of {\idem}s corresponding to {\lagsph}s forming an $A_2$-{\config}. Although {\thm}s \ref{Am config many qmor} and \ref{Dehn twist spec inv} stem from different perspectives compared to {\thm}s \ref{no DE config} and \ref{no DE config SG} (where {\thm} \ref{Am config many qmor} has some {\hamil} dynamical flavor and {\thm} \ref{Dehn twist spec inv} comes from the filtered $A_{\infty}$-categorical perspective), their proofs are based on the spectral rigidity of {\lagsph}s that was studied to prove {\thm} \ref{no DE config SG}.

\subsection{Acknowledgements}
My special thanks go to Jonathan Evans and Kazushi Ueda; they have kindly answered many questions, shared with me a lot of their knowledge, and discussions I had with them at different stages of this work helped me a lot to develop intuition and to improve my understanding. The constructive comments given by Kazushi Ueda on the earlier version of the paper improved the exposition. Many other thanks are due: The discussion with Michel Brion on {\thm} \ref{no DE config} was crucial to understand better the mathematical context; Ailsa Keating and Yuhan Sun kindly answered my questions about their excellent works and gave interesting feedback on this work; Paul Biran made me realize the subtlety in translating the symplectic statement {\thm} \ref{no DE config SG} to the algebro-geometric statement {\thm} \ref{no DE config}, and I received a lot of help from Olivier Benoist, Samir Canning, Olivier de Gaay Fortman for this translation; A discussion with Cheuk Yu Mak lead to the generic version of {\thm} \ref{no DE config}; Fumihiko Sanda communicated to the author about the technical constraints; Osamu Fujino, Hiroshi Iritani,  Kaoru Ono, and Rahul Pandharipande gave interesting and useful feedback from the algebro-geometric point of view, especially on the historical remarks. Lastly, I thank the anonymous referee for a careful reading and valuable remarks. This work was supported by the Hermann-Weyl-Instructorship at the Institute for Mathematical Research (FIM) of ETH Z\"urich.

\section{Preliminaries}

\subsection{Spectral invariant theory}\label{prelim specinv}
 
It is well-known that on a closed {\symp} {\mfd} $(X,\omega)$\footnote{Although the results in this section hold for general closed {\symp} {\mfd}s, we will only be using the monotone case due to some Floer-theoretic constraints that will appear later, which is not from the {\specinv} theory.}, for a non-{\degen} {\hamil} $H:=\{H_t:X \to \R \}_{t\in [0,1]}$ and a choice of a nice {\coeff} field $\Lambda^{\downarrow}$, such as the downward Laurent {\coeff}s $\Lambda_{\text{Lau}} ^{\downarrow} $ for the monotone case
$$\Lambda_{\text{Lau}} ^{\downarrow} :=\{\sum_{k\leq k_0 } b_k t^k : k_0\in \mathbb{Z},b_k \in \mathbb{C} \}  ,   $$
or the downward Novikov {\coeff}s $\Lambda_{\text{Nov}} ^{\downarrow}$ for the general case
$$\Lambda_{\text{Nov}} ^{\downarrow}:=\{\sum_{j=1} ^{\infty} a_j T^{\lambda_j} :a_j \in \mathbb{C}, \lambda _j  \in \mathbb{R},\lim_{j\to -\infty} \lambda_j =+\infty \} ,$$
one can construct a filtered Floer homology group $\{HF^\tau(H):=HF^\tau(H;\Lambda^{\downarrow})\}_{\tau \in \R}$. Note that in this paper, we only use Novikov {\coeff}s, i.e. 
$$ \Lambda ^{\downarrow}=\Lambda_{\text{Nov}} ^{\downarrow} .$$
For two numbers $\tau<\tau'$, the groups $HF^\tau(H;\Lambda^{\downarrow})$ and $HF^{\tau'}(H;\Lambda^{\downarrow})$ are related by a map induced by the inclusion map on the chain level:
$$i_{\tau,\tau'}: HF^\tau(H;\Lambda^{\downarrow}) \longrightarrow HF^{\tau'}(H;\Lambda^{\downarrow}) ,$$
and especially we have
$$i_{\tau}: HF^\tau(H;\Lambda^{\downarrow}) \longrightarrow HF^{}(H;\Lambda^{\downarrow}) ,$$
where $HF^{}(H;\Lambda^{\downarrow})$ is the Floer homology group. There is a canonical ring isomorphism called the Piunikhin--Salamon--Schwarz (PSS)-map \cite{[PSS96]}, \cite{[MS04]}
$$PSS_{H; \Lambda_{} } : QH(X,\omega ;\Lambda_{}) \xrightarrow{\sim} HF(H;\Lambda_{} ^{\downarrow}) ,$$
where $QH(X,\omega;\Lambda_{})$ denotes the quantum cohomology ring of $(X,\omega)$ with $\Lambda$-{\coeff}s, i.e.
$$ QH(X,\omega;\Lambda_{}) := H^\ast (X;\C) \otimes \Lambda_{} .$$
Here, $\Lambda_{}$ is the Novikov {\coeff}s (the universal Novikov field) $\Lambda_{\text{Nov}}$
$$\Lambda_{\text{Nov}}:=\{\sum_{j=1} ^{\infty} a_j T^{\lambda_j} :a_j \in \mathbb{C}, \lambda _j  \in \mathbb{R},\lim_{j\to +\infty} \lambda_j =+\infty \} .$$ 
From now on, we will always take the the universal Novikov field to set-up the {\QH}, so we will often abbreviate it by $QH(X,\omega)$, i.e. 
$$QH(X,\omega) := QH(X,\omega ;\Lambda_{\text{Nov}}) .$$

The ring structure of $QH(X,\omega)$ is given by the quantum product, which is a quantum deformation of the intersection product
$$- \ast - :QH(X,\omega) \times QH(X,\omega) \to QH(X,\omega) . $$

The {\specinv}s, which were introduced by Schwarz \cite{[Sch00]} and developed by Oh \cite{[Oh05]} following the idea of Viterbo \cite{[Vit92]}, are real numbers $\{c (H,a ) \in \R\}$ associated to a pair of a {\hamil} $H$ and a class $a \in QH(X,\omega )$ in the following way:
$$c(H,a ) := \inf \{\tau \in \R : PSS_{H; \Lambda_{} } (a) \in \Im (i_{\tau})\} .$$

\begin{remark}
Although the Floer homology is only defined for a non-{\degen} {\hamil} $H$, the {\specinv}s can be defined for any {\hamil} by using the following \textit{Hofer continuity property}:
\begin{equation}\label{Hofer conti abs}
   \int_{0} ^1 \min_{x\in X} \left(H_t(x) - G_t(x) \right) dt \leq  c(H,a )-c(G,a ) \leq \int_{0} ^1 \max_{x\in X} \left(H_t(x) - G_t(x) \right) dt
\end{equation}
for any $a \in QH(X,\omega),\ H$ and $G$.
\end{remark}

Spectral invariants satisfy the \textit{triangle inequality}: for {\hamil}s $H,G$ and $a,b \in QH(X,\omega)$, we have
\begin{equation}\label{triangle ineq}
    c(H,a) + c(G, b) \geq c(H \# G, a \ast b )
\end{equation}
where $H \# G (t,x):=H_t(x) + G_t( \left( \phi^t _{H} \right)^{-1} (x)  )$ and it generates the path $t \mapsto \phi^t _H \circ \phi^t _G$ in $\Ham(X,\omega)$.

When we take the zero function as the {\hamil}, we have the \textit{valuation property}: for any $a\in  QH(X;\Lambda) \backslash \{0\},$
\begin{equation}\label{valuation}
   c (0,a)=\nu(a) 
\end{equation}
where $0$ is the zero-function and $\nu : QH(X;\Lambda) \to \R$ is the natural valuation function
\begin{equation}
    \begin{gathered}
        \nu : QH(X;\Lambda) \to \R \\
        \nu(a):= \nu(\sum_{j=1} ^{\infty} a_j T^{\lambda_j}):=\min\{\lambda_j: a_j \neq 0\}.
    \end{gathered}
\end{equation}

Note that from the triangle inequality \eqref{triangle ineq} and the valuation property \eqref{valuation}, for any $a\in  QH(X;\Lambda) \backslash \{0\},\ \lambda \in \Lambda$ and a {\hamil} $H$, we have
\begin{equation}\label{novikov shift}
    c(H, \lambda \cdot a) =  c(H,   a) + \nu (\lambda) .
\end{equation}

Analogous invariants for {\lag} Floer homology, namely the {\lag} {\specinv}s, were defined in \cite{[Lec08],[LZ18],[FOOO19],[PS]}. We summarize some basic properties of {\lag} {\specinv}s from these references. Once again, given a pair of a (non-{\degen}) {\hamil} $H$ and a class $a \in HF(L)$\footnote{The {\lag} Floer homology for $L$ without a {\hamil} term $HF(L)$ stands for the {\lag} quantum cohomology \cite{[BC08]}, which is also written as $QH(L)$ in the literature.}, we define 
$$\ell (H,\alpha ) := \inf \{\tau \in \R : PSS_{L,H } (\alpha) \in \Im (i_{\tau} ^L)\} $$
where 
$$PSS_{L,H } : HF(L) \rightarrow HF(L,H),$$
$$i_{\tau} ^L: HF^\tau (L,H) \to HF^\tau (L,H)  .$$
In this paper, we pay particular attention to the case where $\alpha= 1_L$. In this case, we simply denote 
$$\ell_L (H): = \ell (H,1_L ) .$$

Analogously to the {\hamil} case (c.f. \eqref{Hofer conti abs}), we have the \textit{{\lag} control property} for $\ell_L$:
\begin{equation}\label{Lag control prop}
   \int_{0} ^1 \min_{x\in L} H_t(x) dt \leq \ell_L (H)  \leq \int_{0} ^1 \max_{x\in L} H_t(x) dt
\end{equation}

Properties analogous to \eqref{triangle ineq}, \eqref{valuation}, \eqref{novikov shift} also hold for {\lag} {\specinv}s.

Note that both {\hamil} and {\lag} {\specinv}s satisfy the homotopy invariance, i.e. if two normalized {\hamil}s $H$ and $G$ generate homotopic {\hamil} paths $t \mapsto \phi_H ^t$ and $t \mapsto \phi_G ^t$ in $\Ham(X,\omega)$, then 
$$c(H, -) =c(G, -) . $$
Thus, one can define {\specinv}s on $\wt{\Ham} (X,\omega)$:
\begin{equation}
\begin{gathered}
 c: \wt{\Ham} (X,\omega) \times QH(X,\omega) \rightarrow \R \\
 c(\wt{\phi},a):= c(H ,a)
\end{gathered}
\end{equation}
where the path $t \mapsto \phi_H ^t$ represents the class of paths $\wt{\phi}$. Similarly, one can define
$$\ell : \wt{\Ham} (X,\omega) \times HF(L) \rightarrow \R . $$

{\hamil} and {\lag} Floer homologies are related by the closed-open and open-closed maps
\begin{equation}
\begin{gathered}
\mathcal{CO}^0 : QH(X,\omega) \to HF(L),\\
\mathcal{OC}^0 :  HF(L) \to QH(X,\omega) ,
\end{gathered}
\end{equation}
 which are defined by counting certain holomorphic curves. The closed-open map $\mathcal{CO}^0$ is a ring homomorphism and the open-closed map $\mathcal{OC}^0$ defines a module action. As they are defined by counting certain holomorphic curves, which have positive $\omega$-energy, they have the following effect on {\specinv}s.
 
 \begin{prop}[{\cite{[BC08],[LZ18],[FOOO19]}}]\label{co oc spec inv}
 Let $H$ be any {\hamil}. 
 \begin{enumerate}
 \item For any $a \in QH(X,\omega)$, we have
 $$c(H,a) \geq \ell (H , \mathcal{CO}^0 (a) ) .$$
\item For any $\alpha \in HF(L)$, we have
 $$\ell (H , \alpha ) \geq  c(H,\mathcal{OC}^0 (\alpha)) .$$
 \end{enumerate}
 \end{prop}

It is a basic property of the open-closed map (c.f. \cite[Section 2.5.2]{BC12}) that 
\begin{equation}
    \mathcal{OC}^0 (1_L) = [L],
\end{equation}
and thus, by {\propo} \ref{co oc spec inv}, we have
\begin{equation}\label{L spec inv relation}
    \ell (H, 1_L ) \geq  c(H,[L]) .
\end{equation}

\subsection{{\EP} {\qmor}s and (super)heaviness}\label{EP qmors and superheaviness}

Based on {\specinv}s, {\EP} built two theories, namely the theory of (Calabi) {\qmor}s and the theory of (super)heaviness, which we briefly review in this section. 

\textit{Quasimorphisms.} {\EP} constructed a special map on $\wt{\Ham}(X,\omega)$ called the {\qmor} for under some assumptions. Recall that a {\qmor} $\mu$ on a group $G$ is a map to the real line $\R$ that satisfies the following two properties:
\begin{enumerate}
    \item There exists a constant $C>0$ such that 
    $$|\mu(f \cdot g) -\mu(f)-\mu(g)|<C $$
    for any $f,g \in G$.
    
    \item For any $k \in \Z$ and $f \in G$, we have
    $$\mu(f^k)=k \cdot \mu(f) .$$
\end{enumerate} 

The following is {\EP}'s construction of {\qmor}s on $\wt{\Ham}(X,\omega)$.

\begin{theo}[{\cite{[EP03]}}]\label{EP qmor}
Suppose $QH(X,\omega ;\Lambda)$ has a field factor, i.e. 
$$ QH(X,\omega) = Q \oplus A $$
where $Q$ is a field and $A$ is some algebra. Decompose the unit $1_{X}$ of $QH(X,\omega )$ {\wrt} this split, i.e. 
$$1_{X}=e + a .$$
Then, the {\asympt} {\specinv} of $\wt{\phi}$ {\wrt} $e$ defines a {\qmor}, i.e.
\begin{equation}
    \begin{gathered}
        \zeta_{e}:\wt{\Ham}(X,\omega) \longrightarrow \R \\
         \zeta_{e} ( \wt{\phi}) := \lim_{k \to +\infty} \frac{c (\wt{\phi} ^{ k},e )}{k} =\lim_{k \to +\infty} \frac{c (H^{\# k},e )}{k}
    \end{gathered}
\end{equation}
where $H$ is any mean-normalized {\hamil} such that the path $t \mapsto \phi_H ^t $ represents the class $\wt{\phi}$ in $\wt{\Ham}(X,\omega)$.
\end{theo}

\begin{remark}
By slight abuse of notation, we will also see $\zeta_{e}$ as a function on the set of time-independent {\hamil}s:
\begin{equation}
    \begin{gathered}
        \zeta_{e}:C^{\infty}(X) \longrightarrow \R \\
         \zeta_{e} ( H) := \lim_{k \to +\infty} \frac{c (H^{\# k},e )}{k}.
    \end{gathered}
\end{equation}
\end{remark}

\begin{remark}
  The {\lag} {\specinv}s do not appear in the result of {\EP}, but we define the \textit{asymptotic {\lag} {\specinv}s}, as we will use them later on in the proofs. 

\begin{equation}\label{asymp lag spec inv}
    \begin{gathered}
        \ovl{\ell}_L :\wt{\Ham}(X,\omega) \longrightarrow \R \\
         \ovl{\ell}_L := \lim_{k \to +\infty} \frac{ \ell ( \wt{\phi} ^{ k},1_L )}{k}
    \end{gathered}
\end{equation}  
\end{remark}

As mentioned in the introduction, {\EP} {\qmor}s are useful to study the Hofer geometry. For example, in some cases {\EP} {\qmor}s on $\wt{\Ham}(X,\omega)$ descend to $\Ham(X,\omega)$, e.g. when $X=S^2,\C P^2, S^2 \times S^2$. Denote one by $\zeta_{e}: \Ham(X,\omega) \to \R$. Then, by using the homogeneity of $\zeta_{e}$ and the Hofer Lipschitz continuity, we can prove the Hofer diameter conjecture by

\begin{equation}
\begin{aligned}
 \lim_{k \to +\infty} d_{\text{Hof}}(\id, \phi^k) & \geq \lim_{k \to +\infty} |\zeta_{e}(\phi^k)| =\lim_{k \to +\infty} k \cdot |\zeta_{e}(\phi)| =+\infty.
\end{aligned}
\end{equation}

\textit{Superheaviness.} {\EP} introduced a notion of {\symp} rigidity for subsets in $(X,\omega)$ called (super)heaviness.

\begin{definition}[{\cite{[EP09]},\cite{[EP06]}}]\label{def of heavy}
Take an idempotent $e \in QH(X,\omega )$ and denote the {\asympt} {\specinv} {\wrt} $e$ by $\zeta_{e}$. A subset $S$ of $(X,\omega)$ is called
\begin{enumerate}
    \item $e$-heavy if for any time-independent {\hamil} $H:X \to \R$, we have
$$  \inf_{x\in S} H(x)  \leq \zeta_{e} ( H) , $$

\item $e$-{\suphv} if for any time-independent {\hamil} $H:X \to \R$, we have
$$ \zeta_{e} ( H) \leq \sup_{x\in S} H(x) . $$

\end{enumerate}
\end{definition} 

\begin{remark}
    Note that if a set $S$ is $e$-{\suphv}, then it is also $e$-heavy. 
\end{remark}

The following is an easy corollary of the definition of {\suphvness} which is useful.

\begin{prop}[{\cite{[EP09]}}]\label{suphv constant}
Assume the same condition on $QH(X,\omega )$ as in Theorem \ref{EP qmor}. Let $S$ be a subset of $X $ that is $e$-{\suphv}. For a time-independent {\hamil} $H:X \to \R$ whose restriction to $S$ is constant, i.e. $H|_{S}\equiv r,\ r\in \R$, we have
$$\zeta_e (H)=r .$$
In particular, two disjoint subsets of $(X,\omega)$ cannot be both $e$-{\suphv}.
\end{prop} 

\begin{proof}
    The first part is an immediate consequence of the definition of (super)heaviness. As for the second part, suppose we have two disjoint sets $A,B$ in $(X,\omega)$ that are both $e$-{\suphv}. Consider a {\hamil} $H$ that is
    $$H|_A=0,\ H|_B=1.$$
    Then, by {\suphvness}, we have
    $$1= \inf_{x\in B} H(x) \leq \zeta_{e}(H) \leq \sup_{x \in A} H(x) =0 ,$$
    which is a contradiction. 
\end{proof}

We end this section by giving a criterion for heaviness, proved by {\FOOO} (there are earlier results with less generality, c.f. \cite{[Alb05]}) using the closed-open map
$$\mathcal{CO}^0 : QH (X,\omega) \to HF  (L)  . $$

\begin{theo}[{\cite[Theorem 1.6]{[FOOO19]}}]\label{CO map heavy}

Assume $HF ( L ) \neq 0$. If 
$$\mathcal{CO}^0 (e)\neq 0$$
for an idempotent $e\in QH (X,\omega)$, then $L$ is $e$-heavy.
\end{theo}

\begin{remark}
When $\zeta_e$ is {\homo}, e.g. when $e$ is a unit of a field factor of $QH (X,\omega)$ and $\zeta_e$ is an {\EP} {\qmor}, then heaviness and {\suphvness} are equivalent so Theorem \ref{CO map heavy} will be good enough to obtain the {\suphvness} of $L$.
\end{remark}

\subsection{Semi-simplicity of the {\QH}}\label{Semi-simplicity of QH}

In this section, we review the notion of semi-simplicity of the {\QH}, both in the context of algebraic and {\symp} geometry. Indeed, the semi-simplicity of the {\QH} is an algebraic structure that has been studied and used widely across algebraic geometry \cite{[BM04],[Dub96],[Man99]} and {\symp} topology \cite{[BC],[EP03],[FOOO10]}, even though the definitions are slightly different. See the final paragraph for the conclusion. Recall that the (small) {\QH}\footnote{We emphasize that in this paper, {\QH} always refers to the small one unless mentioned otherwise.} of a {\symp} {\mfd} $(X,\omega)$ is semi-simple when it splits into a direct sum of finite fields:
$$QH(X,\omega) = \bigoplus_{1 \leq j \leq k} Q_j$$
where each $Q_j$ is a field. For examples of {\symp} {\mfd}s whose {\QH}s are semi-simple, we refer to Section \ref{context}.

 \begin{remark}
     \begin{enumerate}
         \item By the definition of semi-simplicity, it follows immediately that if $QH(X,\omega)$ is semi-simple, then there is no nilpotent in $QH(X,\omega)$.

         \item  Also, note that {\symp} {\mfd}s with semi-simple {\QH} provide examples to which one can apply {\EP}'s construction of {\qmor}s ({\thm} \ref{EP qmor}).

         \item In this paper, we take the universal Novikov field $\Lambda$ to define the {\QH}, but the {\QH} can defined by other {\coeff} fields, e.g. the field of Laurent series. The choice of a {\coeff} field does not impact the semi-simplicity \cite[{\propo} 2.1]{[EP08]}.
     \end{enumerate}
 \end{remark}

In fact, Usher proves that the existence of one {\symp} form for which the (small) {\QH} is semi-simple implies semi-simplicity for a generic choice of a {\symp} form \cite[{\propo} 7.11]{[Ush11]} (However, having semi-simple {\QH} for a generic {\symp} form does not imply that for all {\symp} forms we have semi-simple {\QH} \cite{[OT09]}). Usher also proves that this generic semi-simplicity for the (small) {\QH} implies the generic semi-simplicity for the big {\QH} \cite[{\propo} 7.5]{[Ush11]}.\footnote{For the definition of the generic semi-simplicity for the big (resp. small) {\QH}, see \cite[Section 7.2]{[Ush11]} (resp. \cite[Definition 7.4]{[Ush11]}).} In practice, the equivalent properties in \cite[{\thm} 7.8, {\propo} 7.11]{[Ush11]} are easier to understand and are more useful.

Moreover, he proves that the notion of semi-simplicity of the big {\QH} used in the algebraic geometry community, e.g. \cite{[BM04],[Dub96],[Man99]}, which is stated in terms of the Frobenius {\mfd}, is equivalent to the {\symp} definition of the generic semi-simplicity of the big {\QH} \cite[Section 7.3.3]{[Ush11]}.

In summary, the monotone semi-simplicity, namely the assumption of {\thm} \ref{no DE config}, implies the semi-simplicity of the big {\QH} used in the algebraic geometry community. Thus, strictly speaking, the assumption of {\thm} \ref{no DE config} is stronger than the notion of semi-simplicity commonly used in algebraic geometry.

\subsection{Degeneration}\label{Degeneration}
In this section, we review some basics of degeneration. A recommended reference for the topic of this section is \cite{[Eva]}.

\begin{definition}\label{def of degen}
    Let $X$ be a smooth variety. A {\degtion} of $X$ is a flat family $\pi: \mathcal{X} \to \C$ such that 
    \begin{itemize}

        \item The only singular fiber is $X_0:= \pi^{-1}(0)$.

\item Some regular fiber is isomorphic to $X$.

        \item The variety $\mathcal{X}$ is smooth away from the {\sing} locus of $X_0$.
  
   \end{itemize}
\end{definition}

If there is a {\degtion} of a variety $X$ such that the central fiber is $X_0:= \pi^{-1}(0)$, we say that $X$ {\degen}s to $X_0$. Note that up to here, the notion of {\degtion} is completely in the realm of {\AG}, but if we are in a situation where the following is valid, then we can start seeing the variety $X$ as a {\symp} {\mfd}.

\begin{itemize}
    \item There is a relatively ample line bundle $\mathcal{L} \to \mathcal{X}$ {\wrt} $\pi: \mathcal{X} \to \C$ (also said $\pi$-relatively ample line bundle).
\end{itemize}

\begin{remark}\label{gorenstein line bdl}
In the proof of {\thm} \ref{no DE config}, we will consider the anti-canonical bundle $- \Omega_{\mathcal{X}}$ (sometimes also written $- K_{\mathcal{X}}$) on $\mathcal{X}$. We will make this precise, as the variety $\mathcal{X}$ is {\sing}. Recall that if a complex $n$-dimensional variety $Y$ is smooth, then the canonical line bundle is $\Omega_Y:= \bigwedge^{n} T^\ast Y $. When $Y$ is {\sing}, the same object cannot be defined. In order to circumvent this problem, Grothendieck introduced the dualizing complex $\omega_Y ^{\bullet}$, which turns out to be a dualizing sheaf $\omega_Y =\omega_Y ^{\bullet}$ if $Y$ is Cohen--Macaulay. Furthermore, if $Y$ is Gorenstein, then $\omega_Y$ is an invertible sheaf, i.e. a line bundle. Note that if $Y$ is smooth, we have $\omega_Y = K_Y$, thus the dualizing sheaves can be considered as the canonical bundles for {\sing} varieties. 

In this paper, we will only be concerned with {\sing} varieties $\mathcal{X}, X_0 $ that have at most isolated {\hypsurf} {\sing}ities, which are Gorenstein, so we can define the line bundles $ \omega_{\mathcal{X}}, \omega_{X_0}$ as well as their inverse line bundles $- \omega_{\mathcal{X}},- \omega_{X_0}$. For a Gorenstein variety $Y$, we say that it is Fano if the line bundle $- \omega_Y $ is ample. 
\end{remark}

The $\pi$-relative ample line bundle $\mathcal{L} \to \mathcal{X}$ defines the following map:
\begin{equation}
\begin{tikzcd}
X_t \arrow[hookrightarrow]{d}{i_t} \arrow{r}{f_t} &  \{ t \} \times \C P^n \arrow[hookrightarrow]{d}{}  &\\
\mathcal{X} \arrow[r,"f"]  \arrow[d,"\pi"] & \mathbb{P}(\pi_\ast \mathcal{L})= \C \times \C P^n \arrow[dl,"\text{pr}_1"] \\
 \C &  &.
\end{tikzcd}
\end{equation}

The restriction of the form $f^\ast \omega_{\FS}$ on each fiber $X_t,\ t \neq 0$ gives a {\symp} form on it. Moreover, by using the form $f^\ast \omega_{\FS}$ on $\mathcal{X}$, one can define a {\symp} parallel transport for $\pi: \mathcal{X} \to \C$, and from a standard argument, it follows that all the smooth fibers are {\symplectotic} {\wrt} the {\symp} form $\omega_t:=(f|_{X_t})^\ast \omega_{\FS} $ on $X_t$ (c.f. \cite[Lemma 1.1]{[Eva]}).

One point we need to be careful about is that if we are interested in the monotone {\symp} form on $X=X_1$, then we need to find a $\pi$-relatively ample line bundle $\mathcal{L} \to \mathcal{X}$ that restricts to the anti-canonical bundle on $X_t:=\pi^{-1}(t),\ t\neq 0$. When the anti-canonical bundle $-K_X$ of a variety $X$ is ample, i.e. $X$ is Fano, the sections of some power $(-K_X)^{\otimes m},\ m \in \N$ give rise to an embedding
$$f: X \hookrightarrow \mathbb{P} (V), $$
where $V$ is the dual of the space of sections, and the pull-back $f^\ast \omega_{\text{FS}}$ is monotone. To see this,
\begin{equation}\label{Fano is monotone}
    \begin{aligned}
        f^\ast [\omega_{\text{FS}}] &=  f^\ast c_1( \mathcal{O}(1)) \\
        & = c_1( ( (\bigwedge^n T^\ast X)^{-1} ) ^{\otimes m}) \\
        & = m \cdot c_1((\bigwedge^n T^\ast X)^{-1}) \\
        & =  -m \cdot c_1(\bigwedge^n T^\ast X) \\
         & =  m \cdot c_1(TX) .
    \end{aligned}
\end{equation}
This is precisely the point where one needs to be careful about when reducing {\thm} \ref{no DE config} to {\thm} \ref{no DE config SG}. See the proof of {\thm} \ref{no DE config} on this.

\begin{remark}
    In this paper, we only study Fano varieties/monotone {\symp} {\mfd}s, but for general type varieties/negative monotone {\symp} {\mfd}s, there are examples of {\degtion}s in \cite{[ES20],[EU21]} for which the relative canonical bundle is not relatively ample. Thus, to study these {\degtion}s one needs to use a {\symp } form that is not negative monotone.
\end{remark}

\subsection{Isolated hypersurface {\sing}ities}\label{hypsing}
In this section, we gather some facts about the isolated hypersurface {\sing}ities. A recommended reference with more details is \cite[Section 1,2]{[Kea15]}.

As we mentioned in the introduction (Section \ref{context}), isolated hypersurface {\sing}ities have been used to produce {\lagsph}s. The modality is a non-negative integer associated to an isolated hypersurface {\sing}ity, which could be thought of as the number of complex parameters of the miniversal deformation space (we refer to \cite{[Kea15]} for more details). Arnold classified isolated hypersurface {\sing}ities up to modality two. The modality zero case is precisely the simple {\sing}ities, which have three types $A_m$, $D_m\ (m\geq 4)$, and $E_m\ (m=6,7,8)$ and they are locally expressed as
\begin{equation}\label{simple sing poly}
    x_1 ^2 + x_2 ^2 + \cdots + x_{n-1} ^2+ p(y,z)=0
\end{equation}
where $p(y,z)$ is 
$$y^2 + z^{m+1},\ 
     y z ^2 +     z ^{m-1} ,\
 y ^3  +z ^{4} ,\
y ^3 + y z ^{3} ,\
 y^3 + z ^{5} ,$$
for $A_m$, $D_m$, $E_6$, $E_7$, and $E_8$, respectively. Now, suppose the variety $X_0$ has an $A_m$ {\sing}ity. The vanishing cycle of the $A_m$ {\sing}ity in its smoothing $(X,\omega)$ forms a collection of {\lagsph}s 
$$\mathscr{S}_{A_m}:= \{S_j\}_{1 \leq j \leq m}$$
satisfying the following intersection property:
    \begin{equation}\label{intersection Am}
 \# \left( S_i \cap S_j \right) =
    \begin{cases}
        1 \text{ if $|i-j|=1$} \\
        0 \text{ otherwise}.
    \end{cases}
\end{equation}

One can see the intersection pattern of \eqref{intersection Am} in the Dynkin diagram of the type $A_m$ (Figure \ref{ADE dynkin diagram}), where each dot in the diagram corresponds to a {\lagsph} and a segment between two dots implies that the {\lagsph}s corresponding to the dots intersect transversally at one point. We call such a collection of {\lagsph}s an $A_m$ {\config}. Similarly, we define $D_m\ (m\geq 4)$, and $E_m\ (m=6,7,8)$ {\config}s of {\lagsph}s to be the collections of {\lagsph}s that satisfy the intersection patterns of the $D_m$ and $E_m$ type Dynkin diagrams, respectively (Figure \ref{ADE dynkin diagram}). These {\config}s appear as the vanishing cycles of {\sing}ities of type $D_m\ (m\geq 4)$, and $E_m\ (m=6,7,8)$, respectively.

As for the isolated hypersurface {\sing}ities of modality one, which were also classified by Arnold, there are the following three types:

    \begin{enumerate}
        \item (parabolic or simple elliptic {\sing}ities) $\wt{E}_{6},\wt{E}_{7},\wt{E}_{8}$.
\item (hyperbolic {\sing}ities) $T_{p,q,r},$ where
\begin{equation}\label{Tpqr poly}
    \begin{gathered}
        T_{p,q,r} = x_1 ^2 + x_2 ^2 + \cdots + x_{n-2} ^2+ h(x, y,z) ,\\
h(x, y,z) = x^p + y^q +z^r + a xyz, \ a\neq 0 
    \end{gathered}
\end{equation}
with integers $p,q,r$ such that 
$$\frac{1}{p}+\frac{1}{q}+\frac{1}{r} <1 .$$
Note that the three pairs $(p,q,r)=(3,3,3),(2,4,4),(2,3,6)$, which are the solutions to $\frac{1}{p}+\frac{1}{q}+\frac{1}{r} =1 $, are precisely the three parabolic {\sing}ities; 
$$\wt{E}_{6}=T_{3,3,3} ,\ \wt{E}_{7}=T_{2,4,4},\ \wt{E}_{8}=T_{2,3,6}.$$

\item 14 exceptional {\sing}ities.
    \end{enumerate}

Other than the simple {\sing}ities and the three types of modality one {\sing}ities, all the other isolated hypersurface {\sing}ities have modality greater than or equal to two. 

\begin{example}\label{Brieskorn Pham}
    There are other famous classes of {\sing}ities such as the \textit{Brieskorn--Pham {\sing}ities}. The Brieskorn--Pham {\sing}ities are isolated hypersurface {\sing}ities which include a part of simple singularities, parabolic (or simple elliptic) singularities, the 14 exceptional unimodal singularities. Thus, Brieskorn--Pham {\sing}ities are covered in {\thm}s \ref{no DE config}, \ref{no DE config SG}. See \cite{[FU11],[Kea21]} for some {\symp} results related to Brieskorn--Pham {\sing}ities.
\end{example}

In \cite{[Kea15]}, Keating executed a detailed study of the vanishing cycles and the Milnor fibers for these {\sing}ities. We recall the definition of the \textit{Milnor fiber}.

\begin{definition}\label{def Milnor fiber}
    The Milnor fiber of a {\hypsurf} {\sing}ity $h=0$, where $h$ is the polynomial expressing the {\sing}ity (e.g. \eqref{simple sing poly}, \eqref{Tpqr poly}), is the intersection of the affine {\hypsurf} $h^{-1}(t) \subset \C ^{n+1}$ for a small $|t|$ and a small ball $B(0;\varepsilon)$ (the ball of radius $\eps>0$ around the origin), i.e. $h^{-1}(t) \cap B(0;\varepsilon)$.\footnote{The definition depends on the choices of $t$ and $\varepsilon$, but it is unique up to {\diffeo}. From a {\symp} viewpoint, different choices of $t$ and $\varepsilon$ will give non-symplectomorphic Milnor fibers, but they both have completions that are symplectomorphic.}
\end{definition} 

Note that Milnor fibers are Liouville domains. We collect some of Keating's results that will be used in the proof of {\thm} \ref{no DE config}.

\begin{prop}\label{parab dynkin}
     The vanishing cycles of the parabolic {\sing}ities form a {\config} of {\lagsph}s with the intersection property as in the Dynkin diagram of Gabrielov \cite[Figure 4]{[Kea15]}. 
\end{prop}

\begin{prop}[{\cite[Lemma 2.12 (resp. 2.13), {\cor} 2.17]{[Kea15]}}]\label{reduction to parab}
     Take any isolated hypersurface {\sing}ity with positive modality. Then, the Milnor fiber (resp. vanishing cycles) of one of the three parabolic {\sing}ities $\wt{E}_{6},\wt{E}_{7},\wt{E}_{8}$ can be {\symp}ally embedded to the Milnor fiber (resp. the vanishing cycles) of the taken isolated hypersurface {\sing}ity with positive modality.
\end{prop} 

{\propo} \ref{reduction to parab} will be used in the second part of the proof of {\thm} \ref{no DE config}.

\begin{remark}
    Note that {\propo} \ref{reduction to parab} was used by Keating in the first line of the proof of \cite[{\thm} 5.7]{[Kea15]}, which might be instructive for the reader.
\end{remark}

\section{Proofs}\label{Proofs}

\subsection{Spectral rigidity of {\lagsph}s}
In this section, we study properties of {\specinv}s for {\lagsph}s, which will be relevant in the later sections. Note that in this section, we do not assume semi-simplicity but we assume that $(X,\omega)$ is a real $2n$-dimensional closed monotone {\symp} {\mfd} with even $n$ (which means that the results in this section are not directly relevant for the odd $n$ case). The monotonicity is assumed for technical reasons.

In \cite[{\thm} 3.3]{[BM16]}, Biran--Membrez proved that for an even dimensional monotone {\lag} sphere $L$ in $(X,\omega)$ (which forces the real dimension of $(X,\omega)$ to be $2n$ with even $n$) satisfies the following property: the cohomology class $[L] \in QH(X,\omega)$, which is the Poincar\'e dual of the homology class represented by $L$, satisfies the cubic equation 
\begin{equation}\label{cubic equation}
    [L]^3 = 4 \beta_{L} [L]
\end{equation}
for some $\beta_L \in \Lambda$. When $\beta_L \neq 0$, then the cubic equation \eqref{cubic equation} implies that 
\begin{equation}\label{sphere idempotents}
        e^{L} _{\pm}= \pm \frac{1}{4\sqrt{\beta_{L}}} [L] + \frac{1}{8\beta_{L}} [L]^2
    \end{equation} 
gives two orthogonal {\idem}s of $QH(X,\omega)$. In fact, the {\idem}s $e^{L} _{\pm}$ are not only {\idem}s of $QH(X,\omega)$, but are units of field factors of $QH(X,\omega)$.

\begin{claim}[{\cite[{\propo} 5.8]{[San21]}}]\label{sphere unit field factor}
    The {\idem}s $e^{L} _{\pm}$ are units of field factors of $QH(X,\omega)$, i.e. 
    $$e^{L} _{\pm} \cdot QH(X,\omega) =  \Lambda \cdot e^{L} _{\pm}  .$$
\end{claim}

 Although this is already proven in \cite[{\propo} 5.8]{[San21]} (for the monotone case for technical reasons), we explain Claim \ref{sphere unit field factor} with an additional assumption that $QH(X,\omega)$ is semi-simple, which is the situation we consider in the rest of the paper, as in this case the argument is elementary. 
 
\begin{proof}
    We know that $e^{L} _{\pm}$ are {\idem}s of $ QH(X,\omega)$ which we assume to be semi-simple, so what we want to check is that $e^{L} _{\pm}$ are not sums of finer {\idem}s, i.e. $e^{L} _{\pm}=e_1 +e_2 +\cdots $ where $e_j$ are {\idem}s. In Sanda's proof of Biran--Membrez's {\lag} cubic equation (\cite[Proof of {\propo} 5.7]{[San21]}, which works for the monotone case), it is shown that for any {\lagsph} $L$, we have
\begin{equation}\label{PD for sphere HF}
    \begin{gathered}
        \mathcal{CO}^0 \circ  \mathcal{OC}^0 (1_{L}) = 2p_L,\\
        \mathcal{CO}^0 \circ  \mathcal{OC}^0 (p_L) =  \beta_L 1_L,
    \end{gathered}
\end{equation}
where $p_L $ is the Poincar\'e dual of the point class of $HF(L)$, and $\beta_L$ is defined by the equation $p_L ^2 =\beta_L 1_L $. The equations \eqref{PD for sphere HF} imply that the map $\mathcal{CO}^0 \circ  \mathcal{OC}^0 $ is an isomorphism (between $\Lambda$-vector spaces);
\begin{equation}
    \begin{tikzcd}
 & QH(X,\omega) \arrow{dr}{\mathcal{CO}^0}& & \\
 HF(L) \arrow{ru}{\mathcal{OC}^0} \arrow{rr}{\sim} & & HF(L) .
    \end{tikzcd}
\end{equation}
Thus, 
$$\mathcal{OC}^0 : HF(L) \to \Im (\mathcal{OC}^0) $$
is also an isomorphism (between $\Lambda$-vector spaces). We have $\Im (\mathcal{OC}^0) = (e^{L} _{+}+e^{L} _{-}) \cdot QH(X,\omega)$, which follows from the following two facts:
\begin{enumerate}
    \item  $\mathcal{OC}^0(1_L)= 2\sqrt{\beta_{L}} ( e^{L} _{+}-e^{L} _{-}),\  \mathcal{OC}^0(p_L)= 2 \beta_{L}  ( e^{L} _{+} + e^{L} _{-}) .$

      \item $\mathcal{OC}^0$ is a $QH(X,\omega)$-module map, i.e. $\alpha \cdot \mathcal{OC}^0(a) =  \mathcal{OC}^0 ( \mathcal{CO}^0 (\alpha) \cdot a)$ for any $\alpha \in QH(X,\omega)$, $a \in HF(L)$ (combined with $e^{L} _{\pm}\in \Im (\mathcal{OC}^0) $ from the first item).
\end{enumerate}

    \begin{remark}
        Note that in the first item, we used that $[L] = 2\sqrt{\beta_{L}} ( e^{L} _{+}-e^{L} _{-}) $ and $$\mathcal{OC}^0(p_L)= \frac{1}{2} \mathcal{OC}^0(\mathcal{CO}^0 \circ  \mathcal{OC}^0 (1_{L}) ) =\frac{1}{2} \mathcal{OC}^0( 1_L ) \cdot \mathcal{OC}^0 (1_{L})  =\frac{1}{2}[L]^2 = 2 \beta_{L}  ( e^{L} _{+} + e^{L} _{-}) .$$
    \end{remark}

We know that $HF(L) \simeq \Lambda \oplus \Lambda $ as a $\Lambda$-vector space, i.e. $\dim_{\Lambda} HF(L)=2$, and thus, by the isomorphism $\mathcal{OC}^0 : HF(L) \xrightarrow[]{\sim} \Im (\mathcal{OC}^0) $ as $\Lambda$-vector spaces, we have $\dim_{\Lambda} (e^{L} _{+}+e^{L} _{-}) \cdot QH(X,\omega) =2$. This implies that $e^{L} _{\pm}$ cannot further split to finer {\idem}s, i.e. they satisfy $e^{L} _{\pm} \cdot QH(X,\omega) =  \Lambda \cdot e^{L} _{\pm}  .$
\end{proof}

\begin{remark}
    It is useful to keep in mind that, when $\beta_L \neq 0$, the class $[L]$ can be expressed by the two {\idem}s in \eqref{sphere idempotents} as follows:
    \begin{equation}\label{L with idempotents}
        [L] = 2\sqrt{\beta_{L}} ( e^{L} _{+}-e^{L} _{-}) .
    \end{equation}
\end{remark}

\begin{lemma}\label{sphere spec inv relation}
    Let $L$ an even-dimensional monotone {\lag} sphere in a closed monotone {\symp} {\mfd} $(X,\omega)$. Assume $\beta_{L} \neq 0$. The {\specinv}s for $1_{L} \in HF(L)$ and $e^{L} _{\pm} \in QH(X,\omega)$ are related as follows:
    \begin{equation}\label{max of two idem}
        \ovl{\ell}_{L} (H) = \max \{ \zeta_{e^{L} _{+}} (H) , \ \zeta_{e^{L} _{-}} (H) \}
    \end{equation}
    for any {\hamil} $H$. In particular, $L$ is $e^{L} _{\pm}$-{\suphv}, i.e. {\suphv} {\wrt} both $e^{L} _{\pm}$.
\end{lemma}

\begin{proof}[Proof of Lemma \ref{sphere spec inv relation}]
From \eqref{L spec inv relation}, we have
$$ \ell( H , 1_L ) \geq  c (H, [L]) $$
for any {\hamil} $H$. By using \eqref{L with idempotents}, we further get
\begin{equation}\label{sph spec inv ineq 1}
    \begin{aligned}
        \ell( H , 1_L ) & \geq  c (H, [L]) \\
        & = c(H , 2\sqrt{\beta_{L}} ( e^{L} _{+}-e^{L} _{-}) ) \\
        & = c(H ,   e^{L} _{+}-e^{L} _{-} ) + \nu (2\sqrt{\beta_{L}}) .
    \end{aligned}
\end{equation}
Note that the last equality uses \eqref{novikov shift}. By using \eqref{PD for sphere HF}, we can see that the closed-open map 
$$\mathcal{CO}^0 : QH(X,\omega) \longrightarrow HF(L) $$
satisfies 
 \begin{equation}\label{closed open map sphere idem}
        \begin{aligned}
          \mathcal{CO}^0 (e^{L} _{\pm}) &= \pm \frac{1}{4\sqrt{\beta_L}}  \mathcal{CO}^0 \circ  \mathcal{OC}^0 (1_{L}) + \frac{1}{8\beta_L} (\mathcal{CO}^0 \circ  \mathcal{OC}^0 (1_{L}))^2 \\
          & = \pm \frac{1}{4\sqrt{\beta_L}}  2p_{L} + \frac{1}{8\beta_L} (2p_{L})^2 \\
          & = \pm \frac{1}{2\sqrt{\beta_L}}  p_{L} + \frac{1}{8\beta_L} 4 \beta_{L}\cdot  1_{L} \\
          & = \pm \frac{1}{2\sqrt{\beta_L}}  p_{L} + \frac{1}{2}   1_{L} .
        \end{aligned}
    \end{equation}

Thus, we have
$$\mathcal{CO}^0 (e^{L} _{+} +e^{L} _{-}) = 1_L .$$
This implies 
\begin{equation}\label{sph spec inv ineq 2}
    c(H , e^{L} _{+} +e^{L} _{-}) \geq \ell( H , 1_L) .
\end{equation}
Inequalities \eqref{sph spec inv ineq 1} and \eqref{sph spec inv ineq 2} imply 
$$ c(H , e^{L} _{+} +e^{L} _{-}) \geq \ell( H , 1_L) \geq c(H ,   e^{L} _{+}-e^{L} _{-} ) + \nu (2\sqrt{\beta_{L}}) .$$
By homogenizing this, we get  
\begin{equation}\label{sph spec inv ineq 3}
    \zeta_{e^{L} _{+} +e^{L} _{-}} (H ) \leq \ovl{\ell}_L ( H ) \leq \zeta_{e^{L} _{+}-e^{L} _{-} } (H    ) .
\end{equation} 
We claim the following.
\begin{claim}\label{spec inv sum equal}
    We have
    $$\zeta_{e^{L} _{+} +e^{L} _{-}} (H) = \zeta_{e^{L} _{+} -e^{L} _{-}} (H) = \max \{ \zeta_{e^{L} _{+}} (H) , \ \zeta_{e^{L} _{-}} (H) \} $$
    for every {\hamil} $H$.
\end{claim}
  Claim \ref{spec inv sum equal} will be proved shortly after the proof. The inequality \eqref{sph spec inv ineq 3} and Claim \ref{spec inv sum equal} imply 
   $$\ovl{\ell}_L ( H ) = \max \{ \zeta_{e^{L} _{+}} (H) , \ \zeta_{e^{L} _{-}} (H) \} $$
   for any $H$, which completes the proof of \eqref{max of two idem}. Now, by combining \eqref{Lag control prop} and \eqref{max of two idem}, we obtain
   \begin{equation*}
     \max \{ \zeta_{e^{L} _{+}} (H) , \ \zeta_{e^{L} _{-}} (H) \} = \ovl{\ell}_L ( H ) \leq \ell_L ( H )   \leq \int_{0} ^1 \max_{x\in L} H_t(x) dt
   \end{equation*}
   for every {\hamil} $H$, which implies that $L$ is {\suphv} {\wrt} both $e^{L} _{+}$ and $e^{L} _{-}$. This completes the proof of Lemma \ref{sphere spec inv relation}.
\end{proof}

\begin{proof}[Proof of Claim \ref{spec inv sum equal}]
We first prove $  \zeta_{e^{L} _{+} +e^{L} _{-}} (H)  \geq \max \{ \zeta_{e^{L} _{+}} (H) , \ \zeta_{e^{L} _{-}} (H) \} $ for any $H$. By the triangle inequality \eqref{triangle ineq}, we get
\begin{equation*}
    \begin{aligned}
        c(H , e^{L} _{+} +e^{L} _{-} ) + \nu (e^{L} _{\pm}) & =  c(H , e^{L} _{+} +e^{L} _{-} ) + c (0, e^{L} _{\pm}) \\
        & \geq c (H , e^{L} _{\pm}) 
    \end{aligned}
\end{equation*}
for any {\hamil} $H$, and thus
$$ \zeta_{e^{L} _{+} +e^{L} _{-} } (H) \geq   \zeta_{e^{L} _{\pm}} (H) $$ 
for any {\hamil} $H$. Therefore, 
\begin{equation}\label{claim ineq 1}
     \zeta_{e^{L} _{+} +e^{L} _{-} } (H )  \geq   \max \{ \zeta_{e^{L} _{+}} (H) , \ \zeta_{e^{L} _{-}} (H) \}
\end{equation}
for any {\hamil} $H$. 

Next, we prove $  \zeta_{e^{L} _{+} +e^{L} _{-}} (H)  \leq \max \{ \zeta_{e^{L} _{+}} (H) , \ \zeta_{e^{L} _{-}} (H) \}$ for any $H$. The characteristic exponent property of spectral invariants (c.f. \cite[Section 2.6.4]{[EP03]}) implies 
$$ c (H , e^{L} _{+} +e^{L} _{-} )  \leq \max \{ c (H ,e^{L} _{+}) ,\ c (H ,e^{L} _{-}) \},$$
and thus by homogenizing, we obtain
\begin{equation}\label{claim ineq 2}
    \zeta_{e^{L} _{+} +e^{L} _{-}} (H )  \leq \max \{ \zeta_{e^{L} _{+}} (H) , \ \zeta_{e^{L} _{-}} (H) \}
\end{equation}
for any {\hamil} $H$. From \eqref{claim ineq 1} and \eqref{claim ineq 2}, we have
\begin{equation}
    \zeta_{e^{L} _{+} +e^{L} _{-}} (H ) = \max \{ \zeta_{e^{L} _{+}} (H) , \ \zeta_{e^{L} _{-}} (H) \}
\end{equation}
for any {\hamil} $H$. By an analogous argument, one can also prove 
\begin{equation}
    \zeta_{e^{L} _{+}  - e^{L} _{-}} (H ) = \max \{ \zeta_{e^{L} _{+}} (H) , \ \zeta_{e^{L} _{-}} (H) \}
\end{equation}
for any {\hamil} $H$. This completes the proof of Claim \ref{spec inv sum equal}.
\end{proof}

Now, we consider the situation where there are several {\lagsph}s, especially when they are (homologically) intersecting. Biran--Membrez (as well as Sanda) proved that if two {\lagsph}s $L$ and $L'$ are (co)homologically intersecting, i.e. 
$$[L] \cdot [L'] \neq 0,$$
then we have 

\begin{equation}\label{intersection property}
\beta_{L} =\beta_{L'} .
\end{equation}

This implies that if there is an ADE {\config} 
$$\mathscr{S}_{}:=\{S_1,\cdots,S_m\},$$
then all the $\beta_{S_j}$ coincide. In such a case, we simply denote
\begin{equation}\label{beta all coincide}
\beta=\beta_{\mathscr{S}_{}} :=\beta_{S_j}  .
\end{equation}

\subsection{Proof of {\thm} \ref{no DE config SG}}
In this section, we prove {\thm} \ref{no DE config SG}. We separate the proof depending on the parity of $n$.

\begin{proof}[Proof of {\thm} \ref{no DE config SG}: case of even $n$]
We start with an important lemma, which will also be used in the proof of {\thm} \ref{Am config many qmor}.

  \begin{lemma}\label{lemma spheres idemp}
    Assume $\beta_{L},\beta_{L'}\neq 0 $. If $[L] \cdot [L'] \neq 0$, then 
    $$\{e^{L} _{+},\ e^{L} _{-} \} \cap \{e^{L'} _{+},\ e^{L'} _{-} \} \neq \emptyset .$$
    \end{lemma}

    \begin{remark}\label{rmk shared idem}
        In Lemma \ref{lemma spheres idemp}, the {\lagsph}s $L,L'$ do not necessarily have to form an $A_2$ {\config}. However, if they do form an $A_2$ {\config}, then we can say furthermore that only one of the two {\idem}s is shared by $L$ and $L'$, i.e. 
        $$\{e^{L} _{+},\ e^{L} _{-} \} \neq \{e^{L'} _{+},\ e^{L'} _{-} \} .$$
        For this, see the proof of {\propo} \ref{intersection and idempotent}. Also note that if $[L] \cdot [L'] = 0$, then $$\{e^{L} _{+},\ e^{L} _{-} \} \cap \{e^{L'} _{+},\ e^{L'} _{-} \}= \emptyset , $$
    as they are all orthogonal.
    \end{remark}
    
    We postpone the proof of Lemma \ref{lemma spheres idemp} until the end of the proof of {\thm} \ref{no DE config SG}. We argue by contradiction; assume that $(X,\omega)$ contains a $D_4$ {\config} of {\lagsph}s. As in Figure \ref{sphere with three hands}, there is a {\lagsph} $S$ that intersects three other {\lagsph}s $S_1,S_2, S_3$,
    $$|S \cap S_j| =1 ,\ 1 \leq j \leq 3 .$$

   \begin{figure}[h]
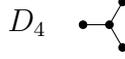

        \tble{D/{4}}
 \caption{$D_4$-{\config}: The sphere $S$ corresponds to the sphere in the middle.}
       \label{sphere with three hands}
   \end{figure}

     Since $n$ is even and $L$ is a {\lagsph}, we have $[L] \cdot [L]=-2$, and thus $[L]\neq 0$. Therefore, by the semi-simplicity of $QH(X,\omega)$, $[L]$ is not a nilpotent, and since $[L]^3 = 4\beta_L [L]$ by the cubic equation \eqref{cubic equation}, we deduce that $\beta_L \neq 0$. By the intersection property \eqref{intersection property}, all the {\lagsph}s involved in the $D_m,E_m$ {\config} have the same $\beta_L$, which is non-zero, i.e. 
    $$\beta:=  \beta_{S}= \beta_{S_j} \neq 0 , \ \forall j. $$
    Thus, by Claim \ref{sphere unit field factor}, each {\lagsph} produces two {\idem}s that are units of field factors of $QH(X,\omega)$ as \eqref{sphere idempotents}:
     
    \begin{equation}
    \begin{gathered}
    e^{S} _{\pm}= \pm \frac{1}{4\sqrt{\beta}} [S] + \frac{1}{8\beta} [S]^2,\\
    e^{S_j} _{\pm}= \pm \frac{1}{4\sqrt{\beta}} [S_j] + \frac{1}{8\beta} [S_j]^2 .
    \end{gathered}
    \end{equation}

From Lemma \ref{lemma spheres idemp}, $S$ and $S_1$ share at least one {\idem}, i.e. 

 \begin{equation}
   \{e^{S} _{+},\ e^{S} _{-}\} \cap \{e^{S_1} _{+},\ e^{S_1} _{-}\} \neq \emptyset .
    \end{equation}
    
    We assume 
     \begin{equation}\label{idem S and S1}
   e^{S} _{-} =  e^{S_1} _{-}
    \end{equation}
    
    {\wlg}. Applying Lemma \ref{lemma spheres idemp} to $S$ and $S_2$, we get 
    
      \begin{equation}
    \begin{gathered}
    e^{S} _{-} \notin  \{e^{S_2} _{+},\ e^{S_2} _{-} \} , \\
    e^{S} _{+} \in  \{e^{S_2} _{+},\ e^{S_2} _{-} \} ,
    \end{gathered}
    \end{equation}
    as 
    \begin{itemize}
        \item $S$ and $S_2$ share at least one {\idem},

        \item $S_1$ and $S_2$ are disjoint, so they cannot share any {\idem} (if they did, then this will contradict the {\suphvness} property {\propo} \ref{suphv constant}),

        \item and we have assumed \eqref{idem S and S1}.
    \end{itemize}

    Once again, {\wlg} we can assume 
    \begin{equation}\label{idem S and S2}
   e^{S} _{+} =  e^{S_2} _{+} .
    \end{equation}
    
   Finally, we apply Lemma \ref{lemma spheres idemp} to $S$ and $S_3$:
   \begin{equation}\label{idem S and S3}
   \{e^{S} _{+},\ e^{S} _{-} \} \cap \{e^{S_3} _{+},\ e^{S_3} _{-} \} \neq \emptyset .
    \end{equation}
    
     Property \eqref{idem S and S3}, combined with \eqref{idem S and S1} and \eqref{idem S and S2}, implies
    \begin{equation}\label{2 neq 3}
        \{e^{S_1} _{-},\ e^{S_2} _{+} \} \cap \{e^{S_3} _{+},\ e^{S_3} _{-} \} \neq \emptyset .
    \end{equation}
 
    In view of {\propo} \ref{suphv constant}, property \eqref{2 neq 3} contradicts the {\suphvness} properties (by Lemma \ref{sphere spec inv relation}, $S_1$ is $e^{S_1} _{-}$-{\suphv}, $S_2$ is $e^{S_2} _{+}$-{\suphv}, and $S_3$ is $e^{S_3} _{\pm}$-{\suphv}), as $S_1$ and $S_2$ are both disjoint from $S_3$, i.e. 
    $$\left( S_1 \cup S_2 \right) \cap S_3 = \emptyset .$$ 
    This completes the proof of {\thm} \ref{no DE config SG} for even $n$.
\end{proof}

We prove Lemma \ref{lemma spheres idemp}.

    \begin{proof}[Proof of Lemma \ref{lemma spheres idemp}]
     First of all, the assumption $[L] \cdot [L'] \neq 0$ implies $\beta_L = \beta_{L'} \neq 0$. We call this $\beta$, i.e.
   $$\beta:= \beta_L = \beta_{L'} \neq 0 .$$

    We have 
    \begin{equation}
        \begin{gathered}
           \frac{1}{2\beta ^{1/2}}[L]= e^{L} _{+}-e^{L} _{-},\\
            \frac{1}{2\beta ^{1/2}}[L']= e^{L'} _{+}-e^{L'} _{-},
        \end{gathered}
    \end{equation}
    and $[L] \ast [L'] \neq 0$ implies
    
    \begin{equation}
        (e^{L} _{+}-e^{L} _{-}) \ast (e^{L'} _{+}-e^{L'} _{-}) \neq 0 .
    \end{equation}
    By developing the left hand side, we get 
    \begin{equation}\label{sum of four terms}
        e^{L} _{+} \ast e^{L'} _{+} -e^{L} _{-} \ast e^{L'} _{+} -e^{L} _{-} \ast e^{L'} _{-} +e^{L} _{-} \ast e^{L'} _{-} \neq 0 .
    \end{equation}
    This implies that at least one of the four terms is non-zero. We can assume 
    $ e^{L} _{+} \ast e^{L'} _{+}\neq 0$ {\wlg}. We prove $e^{L} _{+}= e^{L'} _{+}=e^{L} _{+} \ast e^{L'} _{+} $. As $e^{L} _{+}$ is a unit of a field factor, i.e. $e^{L} _{+} \cdot QH(X,\omega) = \Lambda \cdot e^{L} _{+}$ (Claim \ref{sphere unit field factor}), we have 
    $$e^{L} _{+} \ast e^{L'} _{+} = \alpha \cdot e^{L} _{+}$$
    for some $\alpha \in \Lambda \backslash \{0\}$. The left hand side $e^{L} _{+} \ast e^{L'} _{+}$ is an idempotent, so we have
    $$(\alpha \cdot e^{L} _{+})^2 =  (e^{L} _{+} \ast e^{L'} _{+})^2 = e^{L} _{+} \ast e^{L'} _{+}= \alpha \cdot e^{L} _{+},$$
    which is
    $$\alpha^2 = \alpha,$$
    and thus, $\alpha = 1$. Therefore, 
    $$e^{L} _{+} \ast e^{L'} _{+} = e^{L} _{+}. $$
    By applying the same argument to $e^{L'} _{+}$, we get 
     $$e^{L} _{+} \ast e^{L'} _{+} = e^{L'} _{+}, $$
     which is 
     $$ e^{L} _{+}= e^{L'} _{+}(=e^{L} _{+} \ast e^{L'} _{+} ) .$$
     Thus,
     $$\{e^{L} _{+},\ e^{L} _{-}\} \cap \{e^{L'} _{+},\ e^{L'} _{-}\} \neq \emptyset .$$
    This finishes the proof of the lemma.
    \end{proof}

\begin{proof}[Proof of {\thm} \ref{no DE config SG}: case of odd $n$]
Assume $n$ is odd and $(X,\omega)$ satisfies $\frac{n +1}{2 N_X} \notin \Z$. Then, any {\lagsph} $L$ in $(X,\omega)$ satisfies a $\dim_{\Lambda } HF(L) =2  $ as a vector space (\cite{[BM16]}) and $pt_L ^2 =0$, both for degree reasons. Thus, $1_L$ is the only {\idem} of $HF(L)$. By using the semi-simplicity, we decompose the unit $1_X$ into a sum of units of field factors:
\begin{equation}
    1_X =\sum_{1 \leq j \leq l} e_j .
\end{equation}
As $\mathcal{CO}^0(1_X)=1_L$, $\mathcal{CO}^0 \colon QH(X,\omega) \to HF(L)$ is a ring homomorphism, and $1_L$ is the only {\idem} of $HF(L)$, there exists a unique unit $e_{j_0}$ of a field factor such that
$$\mathcal{CO}^0(e_j)=\delta_{j,j_0} \cdot 1_L , $$
    where $\delta_{j,j_0}=1$ if $j=j_0$ and $\delta_{j,j_0}=0$ otherwise. We denote this distinguished {\idem} corresponding to $L$ by $e^L:=e_{j_0}$.

Now, we consider an $A_2$-{\config} $L,L'$. As they intersect at one point, we have $HF(L,L') \neq 0 (=  \Lambda \cdot \langle L \cap L' \rangle$). Together with $\mathcal{CO}^0(e^L)=1_L$ and $\mathcal{CO}^0(e^{L'})=1_{L'}$, we obtain $e^{L} \ast e^{L'} \neq 0$ (see, for example, \cite[Lemma 4.7]{[San21]}). As $e^{L} , e^{L'} $ are units of fields factors of $QH(X,\omega)$, we conclude that $e^{L} = e^{L'}$ (by the same argument that starts right after the equation \eqref{sum of four terms}). This implies that $L$ and $L'$ are both $e^{L} = e^{L'}$-{\suphv}.

Now, we suppose there is an $A_3$-{\config} $L_1,L_2,L_3$ and deduce a contradiction. The pairs $L_1,L_2$ and $L_2,L_3$ both form $A_2$-{\config}s. Thus, from the previous argument, we have
$$e^{L_1} = e^{L_2}=e^{L_3} .$$
Thus, the three {\lagsph}s $L_1,L_2,L_3$ are all {\suphv} {\wrt} the same {\idem} $e^{L_1} = e^{L_2}=e^{L_3}$. This contradicts that $L_1 \cap L_3 = \emptyset$ ({\propo} \ref{suphv constant}). Thus, there cannot be any $A_3$-{\config} when $n$ is odd and $(X,\omega)$ satisfies $\frac{n +1}{2 N_X} \notin \Z$.
\end{proof}

\begin{remark}
    We can also prove the even $n$ case of {\thm} \ref{no DE config SG} by an argument closer to the above argument for the odd $n$ case.
\end{remark}

\subsection{Proof of {\thm} \ref{no DE config}: From SG to AG}\label{From SG to AG}

In this section, we prove {\thm} \ref{no DE config} by reducing it to its {\symp} counterpart, namely {\thm} \ref{no DE config SG}. Before we start the proof, we mention the following expected statement that extends {\thm} \ref{no DE config}, which will hold as soon as \cite{[AFOOO]} is established.

\begin{conj}[{Algebraic geometry version: expected}]\label{ideal statement}
    Let $X$ be a complex $n$ dimensional smooth Fano variety. Assume either one of the following two:
    \begin{itemize}
        \item $QH(X,\omega)$ is semi-simple, where $\omega$ is the anti-canonical form.

        \item $n>2$ and $QH(X,\omega)$ is semi-simple for a generic choice of a {\symp} form $\omega$.
    \end{itemize}
    If $X$ degenerates to a Fano variety with an isolated {\hypsurf} {\sing}ity, then the {\sing}ity has to be
    \begin{itemize}
        \item  an $A_m$-{\sing}ity with $m\geq 1$, if $n$ is even.

        \item  an $A_m$-{\sing}ity with $m= 1,2$, if $n$ is odd and $\frac{\dim_\C X +1}{2r_X} \notin \Z$ where $r_X$ is the Fano index.
    \end{itemize}
    \end{conj}

\begin{remark}\label{relation two assumptions}

\begin{enumerate}
    \item As we have pointed out in Section \ref{Semi-simplicity of QH}, the two versions of the semi-simplicity that we pose in {\conjec} \ref{ideal statement}, namely the monotone one and the generic one, will imply the semi-simplicity that is commonly used in the community of {\AG}.

\item {\conjec} \ref{ideal statement} does not hold with the the generic semi-simplicity (i.e. the second) assumption when $n=2$. In fact, the del Pezzo surface $\mathbb{D}_5= \C P^2 \# 5 \cdot (\overline{\C P^2})$ is known to have semi-simple {\QH} {\wrt} a generic {\symp} form, but it can {\degen} to a {\sing} del Pezzo surface with a $D_5$ {\sing}ity.
    
\end{enumerate}
\end{remark}

\begin{proof}[Proof of {\thm} \ref{no DE config}]
     As we have pointed out in Section \ref{Degeneration}, in order to reduce {\thm} \ref{no DE config} to {\thm} \ref{no DE config SG}, we need to prove that there is a $\pi$-relative ample line bundle that gives us a monotone {\symp} form on the general fibers (in the neighborhood of the origin).
  
     Suppose $X$ {\degen}s to a Fano variety $X_0$ with {\hypsurf} {\sing}ities, that is, we have a {\degtion} $\pi: \mathcal{X} \to \C$ of $X$ whose central fiber $X_0$ is Fano (in the sense of Remark \ref{gorenstein line bdl}) and has {\hypsurf} {\sing}ities. We claim the following in this situation.

     \begin{claim}\label{relative ampleness near origin}
      The variety $\mathcal{X}/ \C(=\mathcal{X})$ is Gorenstein. Thus, the dualizing sheaf $\Omega _{\mathcal{X}/ \C}=\Omega_{\mathcal{X}}$ is a line bundle over $\mathcal{X}$ (Remark \ref{gorenstein line bdl}). There exists a Zariski-open neighborhood $U \subset \C$ of the origin such that the restriction of $\Omega_{\mathcal{X}}$ to $\mathcal{X}|_{U}=\pi^{-1}(U)$ is $\pi$-relative ample line bundle. 
     \end{claim}

\begin{remark}
    \begin{enumerate}
        \item In this case, $\C$ satisfies $K_{\C}=0$, so we have
        $$\Omega_{\mathcal{X}/ \C} =\Omega_{\mathcal{X} } .$$

        \item A Zariski-open subset of $\C$ is the complement of finitely many points. 
    \end{enumerate}
\end{remark}

    Before we prove Claim \ref{relative ampleness near origin}, we continue with the proof of {\thm} \ref{no DE config}. Take a Zariski-open subset $U$ of $\C$ as in Claim \ref{relative ampleness near origin}. Then, we have a $\pi$-relative ample line bundle $\Omega_{\mathcal{X}}|_{\mathcal{X}|_{U}}$ on $\mathcal{X}|_{U}$. As explained in Section \ref{Degeneration}, this defines a projective embedding 
    $$ f_t: X_t \hookrightarrow \C P^N$$
    for every $t \in U \backslash \{0\}$ which gives us a {\symp} form $\omega_t:=(f_t)^\ast \omega_{\FS}$, where $\omega_{\FS}$ is the Fubini--Study form on $\C P^N$. Now, the {\symp} form $\omega_t$ is a monotone form on $X_t,\ t \neq 0$, as $\Omega_{\mathcal{X}}|_{\mathcal{X}|_{U}} \to \mathcal{X}|_{U}$ is $\pi$-relatively ample, and $\Omega_{\mathcal{X}}|_{X_t} =\Omega_{X_t}=K_{X_t}$, where $K_{X_t}$ is the canonical line bundle (see also \eqref{Fano is monotone}). Now, from the {\degtion}, we obtain a {\config} of {\lagsph}s in $(X_t,\omega_t)$ as the vanishing cycles of the {\hypsurf} {\sing}ities. We deal with simple {\sing}ities and isolated {\hypsurf} {\sing}ities of positive modality separately.

\textit{Case of simple {\sing}ities, i.e. modality zero.}
      First, assume $n$ is even. We argue by contradiction; assume that $X$ degenerates to a Fano variety $X_0$ with a D or E {\sing}ity. Then, as we have discussed above, $(X,\omega)$ contains a D or E {\config} of {\lagsph}s. In either case, as in Figure \ref{sphere with three hands extra}, there is a {\lagsph} $S$ that intersects three other {\lagsph}s $S_1,S_2, S_3$,
    $$|S \cap S_j| =1 ,\ 1 \leq j \leq 3 .$$
    
   \begin{figure}[h]
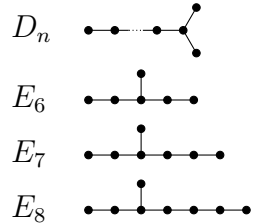

        \tble{D/{},E/{6},E/{7},E/{8}}
 \caption{The sphere $S$ corresponds to, the sphere at the end of the straight line in the $D_n$ diagram, and the third sphere in the $E_6,E_7,E_8$ diagrams, respectively.}
       \label{sphere with three hands extra}
   \end{figure}
By {\thm} \ref{no DE config SG}, these {\config}s do not occur. Thus, DE {\sing}ities cannot occur on $X$. Next, assume $n$ is even and $\frac{\dim_\C X +1}{2r_X} \notin \Z$ where $r_X$ is the Fano index. Similarly to the even $n$ case, by {\thm} \ref{no DE config SG}, an $A_3$-{\config} is prohibited, so $A_1$ and $A_2$ {\sing}ities are the only two simple {\sing}ities that can occur on $X$.

    \textit{Case of positive modality.} 
    The argument goes similarly to the case of modality zero, i.e. simple {\sing}ities. We argue by contradiction; assume that $X$ degenerates to a Fano variety $X_0$ with a positive modality {\sing}ity. Then $(X,\omega)$ contains a {\config} of {\lagsph}s coming from a positive modality {\sing}ity. From {\propo} \ref{reduction to parab}, we know that  the vanishing cycle of an isolated hypersurface {\sing}ity with positive modality includes the vanishing cycle of one of the three parabolic {\sing}ities $\wt{E}_{6},\wt{E}_{7},\wt{E}_{8}$. Keating studied the geometry of the vanishing cycles of the three parabolic {\sing}ities $\wt{E}_{6},\wt{E}_{7},\wt{E}_{8}$ and proved that the {\lagsph}s that appear in the vanishing cycles of the three parabolic {\sing}ities intersect as in the Dynkin diagram of Gabrielov \cite[Figure 4]{[Kea15]}. In the Dynkin diagram of Gabrielov \cite[Figure 4]{[Kea15]}, one can find a collection of four {\lagsph}s $\{S,S_1,S_2,S_3\}$ just as in the case of simple {\sing}ities: there is a {\lagsph} $S$ that intersects three other {\lagsph}s $S_1,S_2, S_3$,
    $$|S \cap S_j| =1 ,\ 1 \leq j \leq 3 .$$

\begin{remark}
    The {\lagsph}s $S,S_1,S_2,S_3$ correspond to the {\lagsph}s 4,1,2,3, respectively, in Keating's numbering in \cite[Figure 4]{[Kea15]}. 
\end{remark}

    The rest of the proof is exactly the same as the modality zero case. We complete the proof of {\thm} \ref{no DE config} by proving Claim \ref{relative ampleness near origin}. As we assume that $X_0$ is Fano, which is equivalent to the ampleness of the line bundle $\Omega_{X_0}=\Omega_{\mathcal{X}/ \C}|_{X_0}$, Claim \ref{relative ampleness near origin} is a direct consequence of \cite[Th\'eor\`eme 4.7.1]{[Gro61]} (see also \cite[{\thm} 1.2.17]{[Laz04]}). Note that \cite[Th\'eor\`eme 4.7.1]{[Gro61]} requires the morphism $\pi: \mathcal{X} \to \C$ to be proper and $\Omega_{\mathcal{X}} $ to be a line bundle. The properness of $\pi$ is satisfied, as all the fibers of $\pi: \mathcal{X} \to \C$ are compact. As for verifying that the dualizing sheaf $\Omega_{\mathcal{X}} $ is indeed a line bundle, it suffices to show that $\mathcal{X}$ is Gorenstein (see Remark \ref{gorenstein line bdl}), and this follows from the following proposition.

    \begin{prop}\label{gorenstein morphism}
        Let $X,Y$ be varieties and $\pi: X \to Y$ a flat morphism. If $Y$ is Gorenstein and all the fibers $\pi^{-1}(y),\ y \in Y $ are Gorenstein, then $X$ is also Gorenstein.
    \end{prop}
In our case
\begin{itemize}
    \item $X_t,\ t\neq 0$ is a smooth variety, and smooth varieties are Gorenstein,
    \item $X_0$ has at most isolated {\hypsurf} {\sing}ities, so it is Gorenstein,
\end{itemize}
and thus all the fibers of the flat morphism of the {\degtion} $\pi: \mathcal{X} \to \C$ are Gorenstein. Thus, $\Omega_{\mathcal{X}} $ is a line bundle from {\propo} \ref{gorenstein morphism}. We have completed the proof of {\thm} \ref{no DE config}.
   
\end{proof}

\begin{remark}\label{Cyclic quotient remark}
    It would be very interesting to study the spectral rigidity of {\sing}ities that are not isolated {\hypsurf} {\sing}ities, for example cyclic quotient {\sing}ities. They are `similar' to the $A_m$-{\sing}ities in the sense that their vanishing cycles are $A_m$-{\config}s attached to a certain {\sing} {\lag} called the {\lag} pinwheel, c.f. \cite[Section 1.2.3]{[Eva]}. In this case, it is interesting to see if a property similar to Lemma \ref{lemma spheres idemp} would hold between a {\lag} pinwheel and a {\lagsph}. Studying the spectral rigidity of {\lag} pinwheels is also an interesting topic. If one could prove results similar to Lemma \ref{sphere spec inv relation} for {\lag} pinwheels, it will bring new applications to Hofer geometry.
\end{remark}

\subsection{Proof of {\thm} \ref{Am config many qmor}}
In this section, we prove {\thm} \ref{Am config many qmor}.

    \begin{proof}[Proof of {\thm} \ref{Am config many qmor}]
    Denote the {\lag} spheres forming the $A_m$-{\config} by 
    $$\mathscr{S}_{A_m}:=\{S_1,\cdots,S_m\}.$$
    As $QH(X,\omega)$ is semi-simple, there is no nilpotent, thus any {\lag} sphere $L$ with $[L] \neq 0 \in QH (X)$ is not a nilpotent. This implies $\beta_L \neq 0$, where $\beta_L$ is the scalar in the cubic equation. Thus, for the $A_m$-{\config} $\mathscr{S}_{A_m}$ we have
    $$\beta= \beta_{\mathscr{S}_{A_m}} = \beta_{S_j} \neq 0 $$
    by \eqref{beta all coincide}. Thus, by \eqref{sphere idempotents}, each {\lagsph} $S_j$ gives rise to two {\idem}s $e^{S_j} _{\pm}$, where
    \begin{equation}
        e^{S_j} _{\pm}= \pm \frac{1}{4\sqrt{\beta}} [S_j] + \frac{1}{8\beta} [S_j]^2.
    \end{equation} 
    
     From Lemma \ref{sphere spec inv relation}, we have for every $j$,
     \begin{equation}
        \ovl{\ell}_{S_j} (H) = \max \{ \zeta_{e^{S_j} _{+}} (H),\ \zeta_{e^{S_j} _{-}} (H)\}
    \end{equation}
    for any {\hamil} $H$, and $S_j$ is $e^{S_j} _{\pm}$-{\suphv} for any $j$. By Lemma \ref{lemma spheres idemp}, the two {\idem}s corresponding to the {\lagsph}s that are next to each other in the $A_m$-{\config}, say $S_j$ and $S_{j+1}$, share one of the two; without loss of generality we can assume
    \begin{equation}
        e^{S_j} _{+}=e^{S_{j+1}} _{-}.
    \end{equation}
    
    To prove Theorem \ref{Am config many qmor}, it is enough to prove the following.
    
    \begin{claim}\label{EP qmor for Am config}
    The {\EP} {\qmor}s 
    $$\{\zeta_{e^{S_j} _{+}}\}_{1\leq j \leq m-1}$$
    are pairwise distinct.
    \end{claim}
    
    We prove this claim. We need to prove that for any $i,j$ such that $1\leq i < j\leq m-1$, we have $\zeta_{e^{S_i} _{+}} \neq \zeta_{e^{S_j} _{+}}$. For such $i,j$, we have that 
    \begin{itemize}
        \item $S_i$ is $e^{S_i} _{+}$-{\suphv},
        \item $S_{j+1}$ is $e^{S_j} _{+}$-{\suphv}, as $e^{S_j} _{+} = e^{S_{j+1}} _{-}$.
    \end{itemize}
    As $j-i \geq 1$, we have $S_i \cap S_{j+1} = \emptyset$. Thus, 
    $$\zeta_{e^{S_i} _{+}} \neq \zeta_{e^{S_j} _{+}}.$$
    This proves Claim \ref{EP qmor for Am config} and thus completes the proof of Theorem \ref{Am config many qmor}.

    \end{proof}

We now prove {\thm} \ref{D4 four qmors}. In \cite[{\thm} C(2)]{[Kaw]}, the author proved the existence of three {\EP} {\qmor}s on $\wt{\Ham}(\mathbb{D}_4)$. The proof uses a set of {\lag} {\config}s arising from a {\tordeg} studied by Y. Sun in \cite{[Sun20]} that consists of a {\lag} torus, an $A_1$-{\config}, and an $A_2$-{\config}. Here, we manage to improve \cite[{\thm} C(2)]{[Kaw]} by completing the union of the $A_1$ {\config} and the $A_2$ {\config} to an $A_4$ {\config} by finding an additional {\lagsph} using some four (real) dimensional technique. This allows us to apply {\thm} \ref{Am config many qmor} and complete the proof.

\begin{proof}[Proof of {\thm} \ref{D4 four qmors}]
First, we find four {\linindep} {\EP} {\qmor}s on $\wt{\Ham}(\mathbb{D}_4)$, the universal cover of $\Ham(\mathbb{D}_4)$. By considering the degeneration of the del Pezzo surface $\mathbb{D}_4$ studied in \cite[Appendix B, case of $X_6$]{[Sun20]}, one obtains a monotone {\lag} torus $L$, an $A_2$ {\config} $\mathscr{S}_{A_2}:=\{S_1,S_2\} $, an $A_1$ {\config} $\mathscr{S}_{A_1}:=\{S_3\} $, that are mutually disjoint (i.e. $L \cap S_1 = \emptyset$, $L \cap \left( S_2 \cup S_3 \right) = \emptyset$, $S_1 \cap \left( S_2 \cup S_3 \right) = \emptyset$). By looking at the moment polytope of the degeneration, one can see that the homology classes represented by the $A_2$ {\config} $\mathscr{S}_{A_2}$ and the $A_1$ {\config} $\mathscr{S}_{A_1}$ are as follows \footnote{This was communicated to the author by Yuhan Sun.} :
\begin{equation}
    \begin{gathered}
        [S_1]= E_2 -E_3, \\
        [S_2] = E_3-E_4 ,\\
        [S_3] = H - E_2 -E_3 -E_4 .
    \end{gathered}
\end{equation}
Now, consider the class $E_4-E_1$. By \cite[Lemma 5.3]{[BLW14]}, one can take a {\lagsph} $S$ that represents the class $E_4-E_1$, i.e. 
$$[S] = E_4-E_1 , $$
that is disjoint from the {\lagsph} $S_1$, i.e. $S_1 \cap S =\emptyset$. By computing the intersection number, one can see that the {\lagsph}s $S_1,S_2,S,S_3$ form a \textit{partial $A_4$ {\config}} in the sense that they satisfy the intersection property \eqref{intersection Am} expect for $S$, which instead satisfies the following weaker, homological version of the intersection property
\begin{equation}
 [S] \cdot [S_j]=
    \begin{cases}
        1 \text{   if $j=2,3$} \\
        0 \text{   if $j=1$} .
    \end{cases}
\end{equation}
However, this is enough to obtain three {\EP} {\qmor}s 
$$\zeta_{e^{S_1} _+},\zeta_{e^{S_2} _+},\zeta_{e^{S} _+}=\zeta_{e^{S_3} _-}$$
which come from the {\lagsph}s involved in the partial $A_4$ {\config}, and they are {\linindep}, as the {\lagsph}s $S_1,S_2,S,S_3$ above satisfy $S_1 \cap S = \emptyset,\ S_2 \cap S_3 = \emptyset$. This is enough to conclude $\zeta_{e^{S_1} _+} \neq \zeta_{e^{S_2} _+},\ \zeta_{e^{S_2} _+} \neq \zeta_{e^{S} _+}$. As it was shown in \cite[Proof of {\thm} C]{[Kaw]}, there is an {\EP} {\qmor} $\zeta_{e^L}$ for which the monotone {\lag} torus $L$ is $e^L$-{\suphv}. As the {\lag} torus $L$, the $A_2$ {\config} $\mathscr{S}_{A_2}$, and the $A_1$ {\config} $\mathscr{S}_{A_1}$ are all mutually disjoint, we conclude that 
\begin{equation}\label{D4 EP qmors}
    \{\zeta_j\}_{1 \leq j \leq 4} :=\{ \zeta_{e^L}, \zeta_{e^{S_1} _+},\zeta_{e^{S_2} _+},\zeta_{e^{S} _+}=\zeta_{e^{S_3} _-} \}
\end{equation}
are four {\linindep} {\EP} {\qmor}s on $\wt{\Ham}(\mathbb{D}_4)$. Evans proved that $\Symp_0 (\mathbb{D}_4) = \Ham(\mathbb{D}_4)$ is weakly contractible, i.e. $\pi_k (\Ham(\mathbb{D}_4)) = \{\id\}$ for all $k \in \N$ \cite[{\thm} 1.3]{[Eva11]}. Thus, any {\qmor} on $\wt{\Ham}(\mathbb{D}_4)$ descends to a {\qmor} on $\Ham(\mathbb{D}_4)$ and thus, $\{\zeta_j\}_{1 \leq j \leq 4}$ defines four pairwise distinct {\EP} {\qmor}s on $\Ham(\mathbb{D}_4)$. As each of the four {\EP} {\qmor}s is Hofer Lipschitz continuous, we conclude that $\R ^4$ embeds quasi-isometrically to $\Ham(\mathbb{D}_4)$, which in particular answers the Kapovich--{\pol} question for $X=\mathbb{D}_4$ in the negative. The three {\qmor}s that are $C^0$ and Hofer-Lipschitz continuous can be obtained as in \cite[{\thm} 22]{[Kaw22]}. This completes the proof of {\thm} \ref{D4 four qmors}.
\end{proof}

\begin{remark}\label{six qmor}
    It seems very likely that one can furthermore prove that {\qmor}s $\zeta_{e^{S_1} _-},\ \zeta_{e^{S_3} _+}$ are pairwise distinct to the {\qmor}s in \eqref{D4 EP qmors} and obtain six pairwise distinct {\qmor}s on $\Ham (\mathbb{D}_4)$. We will just briefly outline the argument: given an $A_m$-{\config} $\{L_j\}_{1 \leq j \leq m}$ in a del Pezzo surface, one can construct a new {\lagsph} $L_{m+1}$ that completes the $A_m$ {\config} into a $m+1$-gon by taking a {\lagsph} representing the class $[L_{m+1}]=-(\sum_{j=1} ^m [L_j])$ (Seidel considers this {\config} of {\lagsph}s in \cite[Example 1.10]{Sei08}). It should be able to show that $L_{m+1}$ can be taken so that it is disjoint to $\{L_j\}_{2 \leq j \leq m-1}$ by a similar argument to above, which shows that $\zeta_{e^{S_1} _-},\ \zeta_{e^{S_3} _+}$ are pairwise distinct to the {\qmor}s in \eqref{D4 EP qmors} by {\suphvness}.
\end{remark}

\subsection{Proof of {\thm} \ref{Dehn twist spec inv}}

In this section, we prove {\thm} \ref{Dehn twist spec inv}. For the definition and the basic properties of the Dehn twist, we refer to \cite{[Sei97]}.

\begin{proof}[Proof of {\thm} \ref{Dehn twist spec inv}]

\textit{Case of even $n$:}
First, notice that as $\beta$ is a {\symp} invariant and $\tau_{L}  \in \Symp(X)$, we have
$$\beta_{\tau_{L}(L')} = \beta_L .$$
Note that we have $\beta_L \neq 0$, as $QH(X,\omega)$ is semi-simple. We denote them all by $\beta$, i.e.
$$\beta:=\beta_{\tau_{L}(L')} = \beta_L \neq 0 . $$

From \eqref{sphere idempotents}, the {\idem}s induced by $\tau_{L}(L')$ are as follows.

\begin{equation}
    \begin{aligned}
       e^{\tau_{L}(L')} _{\pm} & = \pm \frac{1}{4\sqrt{\beta}} [\tau_{L}(L')] + \frac{1}{8\beta} [\tau_{L}(L')]^2 \\
       &= \pm \frac{1}{4\sqrt{\beta}} (\tau_{L})_\ast[L'] + \frac{1}{8\beta} \left( (\tau_{L})_\ast[L']\right)^2 .
    \end{aligned}
\end{equation}

The Picard--Lefschetz formula implies 
\begin{equation}\label{Picard Lefschetz}
    (\tau_{L})_\ast[L'] = [L'] - (-1)^{n(n-1)/2} ([L] \cdot [L']) [L],
\end{equation}
and we also have
\begin{equation}\label{homology with idem}
    \begin{gathered}
       [L]= 2 \sqrt{\beta} (e^{L} _+ - e^{L} _-),\\
        [L']= 2 \sqrt{\beta} (e^{L'} _+ - e^{L'} _-).
    \end{gathered}
\end{equation}
As {\lagsph}s $\{L,L'\}$ form an $A_2$ {\config}, i.e. $|L \cap L'|=1$, there are two possibilities for the intersection number, i.e. $[L] \cdot [L'] =\pm 1$. We have the following relation between the homological intersection number $[L] \cdot [L']$ and the {\idem}s corresponding to $L$ and $L'$.

    \begin{prop}\label{intersection and idempotent}
     Let $(X,\omega)$ is a $2n$-dimensional monotone {\symp} {\mfd} with even $n$. Suppose there is an $A_2$ {\config} $\{L,L'\}$. The {\idem}s corresponding to $L$ and $L'$, namely $\{ e^{L} _+,\ e^{L} _- \}$ and $\{ e^{L'} _+,\ e^{L'} _- \}$, and the intersection number $[L] \cdot [L'] $ have the following relationship:
    \begin{enumerate}
        \item  $(-1)^{n(n-1)/2}[L] \cdot [L'] =-1 $ if and only if $e^{L} _+ = e^{L'} _-$ or $e^{L} _- = e^{L'} _+$.
        
        \item  $(-1)^{n(n-1)/2}[L] \cdot [L'] = 1 $ if and only if $e^{L} _+ = e^{L'} _+$ or $e^{L} _- = e^{L'} _-$.
        
    \end{enumerate}
    
    \end{prop}

\begin{remark}
    As pointed out in Remark \ref{rmk shared idem}, only one of $\{ e^{L} _+,\ e^{L} _- \}$ and $\{ e^{L'} _+,\ e^{L'} _- \}$ is shared. See the proof of {\propo} \ref{intersection and idempotent} for this.
\end{remark}

We prove {\propo} \ref{intersection and idempotent} at the end of this section, after the proof of {\thm} \ref{Dehn twist spec inv}. First, we assume $(-1)^{n(n-1)/2}[L] \cdot [L'] =-1$. From {\propo} \ref{intersection and idempotent}, this implies that we have either $e^{L} _+=e^{L'} _-$ or $e^{L} _-= e^{L'} _+$. We assume the former, as the proof for the latter goes identically.

Putting equations \eqref{Picard Lefschetz} and \eqref{homology with idem} together, we get 
\begin{equation}\label{Dehn twisted homology}
\begin{aligned}
     (\tau_{L})_\ast[L'] &=[L'] + [L] \\
     & = 2 \sqrt{\beta} \left( (e^{L'} _+ - e^{L'} _-) + (e^{L} _+ - e^{L} _-)\right) \\
     & =  2 \sqrt{\beta} (e^{L'} _+ - e^{L} _-).
\end{aligned}
\end{equation}

From \eqref{Dehn twisted homology}, we get
\begin{equation}\label{Dehn twisted idem}
    \begin{aligned}
       e^{\tau_{L}(L')} _{\pm} & = \pm \frac{1}{4\sqrt{\beta}} [\tau_{L}(L')] + \frac{1}{8\beta} [\tau_{L}(L')]^2 \\
       & = \pm \frac{1}{4\sqrt{\beta}} 2 \sqrt{\beta} (e^{L'} _+ - e^{L} _-) + \frac{1}{8\beta} \left( 2 \sqrt{\beta} (e^{L'} _+ - e^{L} _-) \right)^2 \\
       & = \pm \frac{1}{2}(e^{L'} _+ - e^{L} _-) + \frac{1}{2}(e^{L'} _+ + e^{L} _-) \\
       & = \begin{cases}
          \text{$e^{L'} _+$\quad if $\pm=+$},\\
            \text{$e^{L} _-$\quad if $\pm=-$}.
       \end{cases}
    \end{aligned}
\end{equation}
The equation \eqref{Dehn twisted idem} and Lemma \ref{sphere spec inv relation} imply 
\begin{equation}
    \ovl{\ell}_{\tau_{L}(L')} = \max \{  \zeta_{e^{L} _-} , \zeta_{e^{L'} _+}\}.
\end{equation}
Comparing it with
\begin{equation}
\begin{gathered}
   \ovl{\ell}_{L} = \max \{ \zeta_{e^{L} _+},  \zeta_{e^{L} _-}\},\\
   \ovl{\ell}_{L'} = \max \{ \zeta_{e^{L'} _+},  \zeta_{e^{L'} _-}\},
\end{gathered}
\end{equation}
we obtain
\begin{equation}\label{dehn reduce even}
    \ovl{\ell}_{\tau_{L}(L')} \leq \max \{   \ovl{\ell}_{L},   \ovl{\ell}_{L'}  \}.
\end{equation}
The other case where we have $(-1)^{n(n-1)/2}[L] \cdot [L'] =+1$ can be dealt exactly in the same way.

\textit{Case of odd $n$:} 
As we have seen in the proof of {\thm} \ref{no DE config SG} for the case of odd $n$, for any {\lagsph} $L$, we can find a distinguished unique unit $e^L$, which satisfies 
$$\ovl{\ell}_{L} =   \zeta_{e^L} . $$
We have also seen that for any {\lagsph}s satisfying $HF(L,L')\neq 0$, we have $e^L=e^{L'}$. Now, for an $A_2$-{\config} $L,L'$, we have $HF(L,L') \neq 0$ and $HF(L,\tau_{L}(L') ) \neq 0$, so we have
$$e^L=e^{L'}=e^{\tau_{L}(L') } .$$
Thus,
$$\ovl{\ell}_{\tau_{L}(L')}  =  \ovl{\ell}_{L} =  \ovl{\ell}_{L'} (= \zeta_{e^L}=\zeta_{e^{L'}}=\zeta_{e^{\tau_{L}(L')}}) ,$$
so we have
\begin{equation}\label{dehn reduce odd}
    \ovl{\ell}_{\tau_{L}(L')} =  \max \{   \ovl{\ell}_{L},   \ovl{\ell}_{L'}  \}.
\end{equation}

We have completed the proof of {\thm} \ref{Dehn twist spec inv}.
\end{proof}

We end this section with a proof of {\propo} \ref{intersection and idempotent} which was used in the proof of {\thm} \ref{Dehn twist spec inv}.

     \begin{proof}[Proof of {\propo} \ref{intersection and idempotent}]
   First of all, the assumption $[L] \cdot [L'] \neq 0$ implies $\beta_L = \beta_{L'} \neq 0$. We define
   $$\beta:= \beta_L = \beta_{L'} \neq 0 .$$

   By using
    \begin{equation}\label{expression used in lemma}
        \begin{gathered}
           [L]= 2\beta ^{1/2} (e^{L} _{+}-e^{L} _{-}),\\
           [L']= 2\beta ^{1/2}( e^{L'} _{+}-e^{L'} _{-}),
        \end{gathered}
    \end{equation}
we have

    \begin{equation}\label{A2 config intersection}
        \begin{aligned}
           [L] \cdot [L'] & = \int_{X} 4 \beta (e^{L} _+ - e^{L} _-) \cdot (e^{L'} _+ - e^{L'} _-) \\
          & = 4 \beta \int_{X}  (
          e^{L} _+ \cdot e^{L'} _+ 
          - e^{L} _- \cdot e^{L'} _+ 
          - e^{L} _+ \cdot e^{L'} _+ 
          + e^{L} _- \cdot e^{L'} _- ) .
        \end{aligned}
    \end{equation}

We will show that only one of the four terms survive. This is because only one of the two {\idem}s between $\{e^{L} _{+},\ e^{L} _{-} \}$ and $\{e^{L'} _{+},\ e^{L'} _{-} \} $ is shared, i.e. $\{e^{L} _{+},\ e^{L} _{-} \} \cap \{e^{L'} _{+},\ e^{L'} _{-} \} \neq \emptyset $ and $\{e^{L} _{+},\ e^{L} _{-} \} \neq \{e^{L'} _{+},\ e^{L'} _{-} \}  $. To see this, first from Lemma \ref{lemma spheres idemp}, we know that $\{e^{L} _{+},\ e^{L} _{-} \} \cap \{e^{L'} _{+},\ e^{L'} _{-} \} \neq \emptyset $. However, we also have $\{e^{L} _{+},\ e^{L} _{-} \} \neq \{e^{L'} _{+},\ e^{L'} _{-} \}  $, as if we had $\{e^{L} _{+},\ e^{L} _{-} \} = \{e^{L'} _{+},\ e^{L'} _{-} \}  $, then \eqref{expression used in lemma} implies 
$$[L] = \pm [L'] .$$
On one hand, $L,L'$ forming an $A_2$ {\config} implies 
$$[L] \cdot [L'] = \pm 1,$$
but on the other hand, $L$ being an even-dimensional {\lagsph} implies 
$$[L] \cdot [L'] = [L] \cdot (\pm [L] )= \pm 2 ,$$
which is a contradiction. Thus, we have proven that only one of the two {\idem}s between $\{e^{L} _{+},\ e^{L} _{-} \} $ and $\{e^{L'} _{+},\ e^{L'} _{-} \} $ is shared. From the proof of Lemma \ref{lemma spheres idemp}, when two {\idem}s from $\{e^{L} _{+},\ e^{L} _{-} \} $ and $\{e^{L'} _{+},\ e^{L'} _{-} \} $ do not coincide, then they are orthogonal to each other, so only one of the four terms in \eqref{A2 config intersection} remains.

We have 
\begin{equation}\label{eL integration}
    \begin{aligned}
        \int_{X} e^{L} _{\pm} & = \int_{X} \left( \pm \frac{1}{4\sqrt{\beta}} [L] + \frac{1}{8\beta} [L]^2 \right) \\
        & = \int_{X}  \frac{1}{8\beta} [L]^2 \\
        & = \frac{1}{8\beta} \int_{X}   [L] \cup [L]  \\
        &=\frac{1}{8\beta} [L] \cdot [L] \\
        & = \frac{1}{8\beta} \cdot (-1)^{n(n-1)/2} \chi (L) \\
        &= (-1)^{n(n-1)/2} \cdot \frac{1}{4\beta}.
    \end{aligned}
\end{equation}
Note that the second and the third equality follow from degree reasons, and the fifth equality uses the fact that $L$ is a {\lagsph}, $\chi (L)$ is the Euler characteristic of $L$, i.e. two. The same property applies for $L'$:
\begin{equation}\label{eL' integration}
    \int_{X} e^{L'} _{\pm}  = (-1)^{n(n-1)/2} \cdot \frac{1}{4\beta} .
\end{equation}
Now, it follows from \eqref{A2 config intersection}, \eqref{eL integration}, and \eqref{eL' integration} that the intersection number $[L] \cdot [L'] $ satisfies the following:
\begin{itemize}
    \item $e^{L} _+ = e^{L'} _-$ or $e^{L} _- = e^{L'} _+$ if and only if $[L] \cdot [L'] = -(-1)^{n(n-1)/2} .$

    \item $e^{L} _+ = e^{L'} _+$ or $e^{L} _- = e^{L'} _-$ if and only if $[L] \cdot [L'] = (-1)^{n(n-1)/2} .$
\end{itemize}
This completes the proof of {\propo} \ref{intersection and idempotent}.
        
    \end{proof}

     \begin{remark}
         A recent work of Biran--Cornea \cite[Lemma 5.3.1]{BC21} contains a result that is based on a spirit similar to {\thm} \ref{Dehn twist spec inv}. Although they work in a different setting, namely with exact {\lag}s in a Liouville {\mfd}, it would be interesting to study if one can obtain a similar inequality to {\thm} \ref{Dehn twist spec inv} by their method.
     \end{remark}

\appendix

\Addresses

\end{document}